\documentclass[11pt,reqno]{amsart}
\usepackage[top=90pt,bottom=90pt,left=90pt,right=90pt]{geometry}
\usepackage{listings}
\usepackage{ytableau}
\usepackage{tikz-cd}
\usepackage{tikz}
\usepackage{tikz-3dplot}
\usepackage{pgfplots}
\pgfplotsset{compat=1.18}
\usetikzlibrary{3d}
\usepackage{float}
\usetikzlibrary{calc}

\usepackage{bbm,amsmath,amssymb,amsfonts,graphicx,color,subfig,mathtools, verbatim,MnSymbol}

\usepackage[bookmarksopen=false,pdftex=true,breaklinks=true,%
      hyperindex=true,pdfstartview=FitH,colorlinks=true,%
      pdfpagelabels=true,colorlinks=true,linkcolor=black,%
      citecolor=black,urlcolor=black,hypertexnames=false%
      ]%
   {hyperref}
\usepackage[capitalize,nameinlink,noabbrev]{cleveref}
\usepackage[
  backend=biber,
  style=alphabetic,
  maxbibnames=99, 
  sorting=nyt,    
  giveninits=true 
]{biblatex}

\theoremstyle{definition}
\newtheorem{theorem}{Theorem}[section]
\newtheorem{proposition}[theorem]{Proposition}
\newtheorem{corollary}[theorem]{Corollary}
\newtheorem{lemma}[theorem]{Lemma}

\newtheorem{question}[theorem]{Question}
\newtheorem{definition}[theorem]{Definition}
\newtheorem{example}[theorem]{Example}
\newtheorem{remark}[theorem]{Remark}
\newtheorem{observation}[theorem]{Observation}

\newcommand{\NN}{\mathbb{N}}
\newcommand{\ZZ}{\mathbb{Z}}

\newcommand{\RR}{\mathbb{R}}
\newcommand{\CC}{\mathbb{C}}
\newcommand{\KK}{\mathbb{K}}

\renewcommand{\leq}{\leqslant}
\newcommand{\YT}{\Tab}
\renewcommand{\Re}{\operatorname{Re}}
\renewcommand{\Im}{\operatorname{Im}}

\newcommand{\setn}[1]{\left\{#1\right\}}

\newcommand{\unlhc}%
{%
  \mathrel{%
    \begin{tikzpicture}[baseline]
      \node (n) at (0,0.1) {$\unlhd$};
      \draw ($(n.south west)+(2mm,0.7mm)$) -- ($(n.north east)+(-2mm,-3.3mm)$);
    \end{tikzpicture}%
  }%
}

\definecolor{cyan}{RGB}{65, 138, 127}
\newcommand{\unrhc}%
{%
  \mathrel{%
    \begin{tikzpicture}[baseline]
      \node (n) at (0,0.1) {$\unrhd$};
      \draw ($(n.south west)+(2mm,0.8mm)$) -- ($(n.north east)+(-2mm,-3.2mm)$);
    \end{tikzpicture}%
  }%
}
\tikzstyle{labeledvertex} = [draw=black, thick, fill=white, inner sep=0.25mm, circle]
\tikzstyle{vertex} = [draw=black, fill=black, inner sep=0.6mm, circle]
\tikzstyle{edge} = [black, thick,-Stealth]

\DeclareMathOperator{\len}{len}
\DeclareMathOperator{\spe}{sp}

\DeclareMathOperator{\sgn}{sgn}
\DeclareMathOperator{\Tab}{Tab}

\begin{document}
\ytableausetup{centertableaux}
\title[]
{
Posets for Specht ideals\\ of essential real reflection groups
}
\author{Sebastian Debus}
\address{Technische Universität Chemnitz, Fakultät für Mathematik, 09107 Chemnitz, Germany}
\email{sebastian.debus@mathematik.tu-chemnitz.de}
\author{Kurt Klement Gottwald}
\address{Technische Universität Chemnitz, Fakultät für Mathematik, 09107 Chemnitz, Germany}
\email{kurt-klement.gottwald@mathematik.tu-chemnitz.de}

\begin{abstract}
Specht ideals are symmetric ideals in the polynomial ring generated by Specht polynomials associated with group representations. These ideals were previously studied for reflection groups of types \( A \) and \( B \), where their inclusion relations and their varieties reflect rich combinatorial structures. In this paper, we extend this theory to type $D$ and the dihedral groups.  
Our results complete the combinatorial study of Specht ideals across all infinite families of essential real reflection groups.
\end{abstract}
\maketitle
\markboth{}{}

\section{Introduction and notation}
A classical construction of the irreducible representations of the symmetric group \( S_n \) in characteristic zero is based on Specht polynomials, which were introduced by Wilhelm Specht~\cite{specht1937Adarstellungstheorie}. 
The ideals generated by the Specht polynomials corresponding to a given irreducible representation are known as Specht ideals (also referred to as Garnir ideals). These ideals are symmetric, i.e. invariant under the group action, and exhibit rich combinatorial, algebraic and geometric properties. 
Notably, Specht ideals are radical \cite{Murai2022Specht, woo2005ideals} and contain the entire isotypic component of the associated irreducible representation in the polynomial ring \cite{moustrou2021symmetric, woo2005ideals}. 
The varieties defined by Specht ideals are unions of intersections of reflection hyperplanes, i.e. unions of elements from the braid arrangement, and whether a point lies in a variety is completely determined by the stabilizer subgroup of the point. \\
Beyond their representation theoretic motivation, Specht ideals have proven to be useful in applications such as solving symmetric systems of polynomial equations, where they help bound the possible symmetries of solutions via orbit types \cite{moustrou2021symmetric}. 
Specht ideals also appear in commutative algebra and algebraic geometry. 
The collection of all Specht polynomials contained in a given Specht ideal forms a universal Gröbner basis \cite{Murai2022Specht, woo2005ideals}, and the ideals have connections to the weak Lefschetz property \cite{McDaniel}. 
Moreover, Specht ideals are used in graph theory to study the independence number of graphs \cite{LiLi,Lovasz} and in the study of symmetric Hilbert schemes of points to classify symmetric ideals \cite{debus2024symmetric}. \\
In type \(A\), corresponding to the symmetric group, Specht ideals exhibit a particularly elegant combinatorial structure. Irreducible representations are labeled by integer partitions, and three naturally arising partially ordered sets turn out to be equivalent: the poset of partitions ordered by dominance, the poset of Specht ideals ordered by inclusion, and the poset of Specht varieties ordered by inclusion as well \cite{moustrou2021symmetric,woo2005ideals}. \\
This remarkable correspondence inspired further investigation into Specht ideals for other reflection groups. 
In type \( B \), the hyperoctahedral group, Specht polynomials were also introduced by Wilhelm Specht in his subsequent work~\cite{specht1937Bdarstellungstheorie}. 
Although the underlying combinatorics are more intricate than in type \( A \), analogous combinatorial and structural properties persist. The Specht ideals in type \( B \) exhibit similar inclusion behavior, and their associated varieties are again determined by orbit types \cite{debus_moustrou_riener_verdure}. \\
These developments naturally raise the question of whether similar algebraic and combinatorial structures exist for Specht ideals of other real reflection groups. 
The essential real reflection groups consist of the four infinite families \( A_{n}, B_n, D_n, I_2(n) \) and six exceptional groups \( H_3, H_4, F_4, E_6, E_7, E_8 \) \cite[Chapter~2]{humphreys1992reflection}. 
This paper focuses on the two remaining previously unexplored infinite families: the dihedral groups \( I_2(n) \) and the even signed symmetric groups \( D_n \).
Our results discover new combinatorial phenomena in reflection group settings and our methods combine combinatorial techniques and algebraic geometry. \\
Our main contributions are the following: For a dihedral group we define the Specht ideal of an irreducible representation such that it contains the respective isotypic component in the polynomial ring (\cref{thm:main2dihedral}). 
This is the desired property in analogy with types $A$ and $B$. 
We completely characterize the inclusions (\cref{thm: main1dihedral}), the decomposition of the varieties as unions of reflection hyperplanes and the few cases of radical Specht ideals (\cref{thm:main2dihedral}).
The largest part of the paper is dedicated to analyzing type $D$ Specht ideals. 
We introduce a new combinatorial partial order on dipartitions (\cref{def:DnPoset}), which encodes inclusion relations among Specht ideals and their varieties (\cref{thm:main}).
Additionally, we characterize the Specht varieties by reducing to type $B$ Specht varieties and orbit types (\cref{prop:Specht varieties I}). 
Moreover, we prove that a uniform description of the Specht varieties in terms of orbit types does not exist (\cref{thm:nounionanalogue}).
Thus, not all strong combinatorial properties of Specht ideals and varieties, which hold in type $A$ and $B$, transfer to the remaining essential real reflection groups.

\subsection*{Outline of the Paper}
In \cref{sec:dihedral groups} we investigate Specht ideals for the dihedral groups. 
We introduce notation for tableaux combinatorics and review Specht ideals of types \( A \) and \( B \) in \cref{sec:pre. for D_n case}. We proceed with defining a partial order on dipartitions adapted to the structure of type \( D \) Specht ideals in \cref{sec:Partial order}.
Then \cref{sec:PosetD_n} is devoted to extending the poset combinatorics to type \( D\) Specht ideals.
Finally, in \cref{sec:conclusion remarks} we present brief concluding remarks and open questions.
In the \hyperref[appendix]{Appendix} we present details that we omitted in the main paper. 
We also include code to compute Specht ideals for types \( A, B \) and \( D \).

\subsection*{Notational convention}
We denote by $\NN=\{0,1,2,\ldots\}$ the natural numbers and for $0\neq n \in \NN$ we write $[n]\coloneqq \{1,2,\ldots,n\}$ for the discrete interval from $1$ to $n$. For a complex number $z=a+\mathrm{i}b$, with $a,b \in \RR$ we write $\Re(z)=a$ for its real part and $\Im(z)=b$ for its imaginary part. For a point $x = (x_1,\ldots,x_n)$ in a vector space we denote by $x^2 := (x_1^2,\ldots,x_n^2)$ the coordinatewise square of $x$.

\section{Dihedral Groups}\label{sec:dihedral groups}
In this section we fix $n \in \NN$. 
The \emph{dihedral group} $I_2(n) \subseteq \operatorname{GL}_2(\RR)$ is the symmetry group of a regular $n$-gon and is generated by the reflection $s$
and the rotation
$r$, where
\[ s =  \begin{pmatrix}
   1  & 0 \\
   0  & -1
\end{pmatrix} ,~ r = \begin{pmatrix}
  \cos(\frac{2\pi}{n})   & -\sin (\frac{2\pi}{n}) \\
  \sin (\frac{2\pi}{n})   & \cos(\frac{2\pi}{n})
\end{pmatrix} . \] 
We refer to \cite[Chapter~1-3]{kane} for background on dihedral groups.
The action of $I_2(n)$ on $\RR^2$ induces an action on the polynomial ring $\RR[x,y]$, where
\[s \cdot x = x, s \cdot y = -y \text{ and } r \cdot x = \cos\left(-\frac{2 \pi }{n}\right)x+\sin\left(-\frac{2 \pi}{n}\right)y, r \cdot y = - \sin\left(-\frac{2 \pi}{n}\right)x + \cos\left(-\frac{2 \pi }{n}\right)y.\] 
It is sometimes useful to consider the actions of $s$ and $r$ on $\CC[x,y] \cong \CC \otimes_\RR \RR[x,y]$, so that
$\sigma \cdot \Re (f(x,y)+\mathrm{i}\cdot g(x,y)) = \sigma \cdot f(x,y)$ and $\sigma \cdot \Im (f(x,y)+ \mathrm{i}\cdot g(x,y)) = \sigma \cdot g(x,y)$ for $f,g \in \RR[x,y]$.
Note that
\begin{align*}
    s\cdot (x+\mathrm{i}y) &= x-\mathrm{i}y,\\
    r \cdot (x+\mathrm{i}y)&= \cos\left(-\frac{2\pi}{n}\right)x-\mathrm{i}\cdot\sin\left(-\frac{2\pi}{n}\right)x+ \sin \left(-\frac{2\pi}{n}\right)y+\mathrm{i}\cdot\cos \left(-\frac{2\pi}{n}\right)y = \mathrm{e}^{\mathrm{i}\frac{2\pi}{n}}x+\mathrm{e}^{\mathrm{i} (\frac{2\pi}{n}+\frac{\pi}{2})}y 
\end{align*}
Thus, for every $k \in \NN$ we have
\begin{align}
r \cdot (x+\mathrm{i}y)^k  =
  (\mathrm{e}^{\mathrm{i}\frac{2\pi}{n}}x+\mathrm{e}^{\mathrm{i} (\frac{2\pi}{n}+\frac{\pi}{2})}y)^k & = \sum_{j=0}^k \mathrm{e}^{\mathrm{i} \frac{2\pi}{n}(k-j)+\mathrm{i}(\frac{2\pi}{n}+\frac{\pi}{2})j}{k \choose j} y^j x^{k-j}  \label{eq: x+iy} \nonumber \\
  & = \sum_{j=0}^k \mathrm{e}^{\mathrm{i} \frac{2\pi}{n}k}{k \choose j} \mathrm{i}^jy^j x^{k-j}  \\
  &  = \mathrm{e}^{\mathrm{i} \frac{2\pi}{n}k}(x+\mathrm{i}y)^k. \nonumber
\end{align}

\begin{remark}
    There is also an action of $I_2(n)$ on $\CC^2$ where the group in $\operatorname{GL}_2(\CC)$ is generated by the reflection $s'$ and the rotation $r'$, where
\[ s' = \begin{pmatrix}
    0 & 1  \\
    1  &  0 
\end{pmatrix},~ r' = \begin{pmatrix}
    \mathrm{e}^{\frac{2\pi}{n}} & 0  \\
    0  &  \mathrm{e}^{-\frac{2\pi}{n}} 
\end{pmatrix}. \]
The diagonal coinvariant space of this action has been investigated in \cite{boij2010notes}.
\end{remark}
Recall that for odd $n$ the dihedral group $I_2(n)$ has $2$ one-dimensional irreducible representations and $ \frac{n-1}{2}$ two-dimensional irreducible representations. 
For even $n$ the group $I_2(n)$ has $4$ one-dimensional irreducible representations and $\frac{n-2}{2}$ two-dimensional irreducible representations (see e.g. \cite[Chapter~5.3]{serre1977linear}). \\
In the following, we first work towards finding an explicit realization of each irreducible representation in $\RR[x,y]$. 
We do so by computing the isotypic decomposition of the coinvariant space of the dihedral group $I_2(n)$.
Recall that the \emph{coinvariant space} of $\RR[x,y]$ for the dihedral group $I_2(n)$ is the quotient space $\RR[x,y]/(\RR[x,y]^{I_2(n)})$, where $(\RR[x,y]^{I_2(n)})$ denotes the ideal generated by the \emph{invariant ring}.
Since $\RR[x,y] = \RR[x,y]^{I_2(n)} \otimes_\RR \RR[x,y]/(\RR[x,y]^{I_2(n)})$ we can obtain an isotypic decomposition of $\RR[x,y]$ by tensoring the coinvariant space with trivial representations from the invariant ring.
We refer to \cite[Chapter~8]{bergeron} for background on coinvariant spaces for reflection groups. \\
Recall that for a reflection group $G$ acting induced on an $n$-variate polynomial ring the invariant ring is generated by $n$ algebraically independent polynomials by the Shepard-Todd-Chevalley theorem \cite[Chapter~18-1]{kane}. 
Such $n$ polynomials are called \emph{fundamental invariants} of $G$.
The fundamental invariants are not unique.
For the dihedral group the fundamental invariants are known.\footnote{See e.g. https://planetmath.org/dihedralgroup (last accessed on 05.06.2025).} 
We include a proof for completeness.

\begin{lemma}\label{lem:fundamental invariants}
 The polynomials $\psi_1(x,y)= x^2+y^2$ and $\psi_2 (x,y) = (x+\mathrm{i}\cdot y)^n+(x-\mathrm{i}\cdot y)^n = 2 \Re (x+\mathrm{i}\cdot y)^n$ are fundamental invariants of $I_2(n)$. 
\end{lemma} 
\begin{proof}
First, we verify that $\psi_1,\psi_2$ are indeed invariant: 
\begin{align*}
    s \cdot \psi_1 & = x^2+(-y)^2  = \psi_1,  s \cdot \psi_2  = \Re ((x-\mathrm{i}y)^n) = \Re ((x+\mathrm{i}y)^n) = \psi_2, \\
    r \cdot \psi_1 & = (\cos(\frac{-2\pi}{n})x+\sin(\frac{-2\pi}{n})y)^2+(-\sin(\frac{-2\pi}{n})x+\cos(\frac{-2\pi}{n})y)^2  = x^2+y^2 = \psi_1 \\
    r \cdot \psi_2 & = \Re ( \mathrm{e}^{\mathrm{i}2\pi} (x+\mathrm{i}y)^n) = \Re ((x+\mathrm{i}y)^n) = \psi_2.
\end{align*}   
Second, $\psi_1$ and $\psi_2$ are algebraically independent since the determinant of the Jacobian matrix of $\Psi : \RR^2 \to \RR^2, (x,y) \mapsto (\psi_1(x,y),\psi_2(x,y)), $ which equals
\[ \det \begin{pmatrix}
 2x    & 2y \\
 n(x+\mathrm{i}\cdot y)^{n-1} + n(x-\mathrm{i}\cdot y)^{n-1}    &  \mathrm{i}  n(x+\mathrm{i}\cdot y)^{n-1} - \mathrm{i} n(x-\mathrm{i}\cdot y)^{n-1} 
\end{pmatrix}  = -4n \Im (x+\mathrm{i}y)^n\]
does not vanish (see e.g. \cite[Page~12]{Lefschetz} for the Jacobian criterion).\\
Third, the fundamental invariants of the dihedral group $I_2(n)$ have degrees $2$ and $n$ (\cite[Page~9]{humphreys1992reflection}) which shows that $\psi_1, \psi_2$ are indeed fundamental invariants.
\end{proof}

Let $\Delta(x,y) \coloneqq \Im ((x+\mathrm{i}y)^n) \in \RR[x,y]$, which equals, up to a scalar, the determinant of the Jacobian matrix of $(\psi_1,\psi_2)$ from \cref{lem:fundamental invariants}. 
The coinvariant space of $I_2(n)$, i.e. the vector space $\RR[x,y]/(\psi_1,\psi_2)$, is isomorphic to the regular representation of $I_2(n)$.
Note that the coinvariant space is the quotient space of two vector spaces and we realize the coinvariant space as the \emph{space of harmonics} (see \cite[Chapter~8.2]{bergeron}). 
Then, the vector space $\RR[x,y]/(\psi_1,\psi_2)$ can be realized as the space spanned by all (higher order) partial derivatives of $\Delta$. \\
Although the space of harmonics of the dihedral group $I_2(n)$ depends on $n$, we just write $\mathcal{H}$ for the space of harmonics of $I_2(n)$. 
We do so, because we work with a fixed natural number $n$ in this section.

\begin{lemma}
The space of harmonics $\mathcal{H}$ has the vector space basis 
\[\{1, \Delta, \Re((x+\mathrm{i}y)^k), \Im ((x+\mathrm{i}y)^k) : 1 \leq k \leq n-1\}.\]
\end{lemma}
\begin{proof}
The dimension of $\RR[I_2(n)]\cong \mathcal{H}$ is $2n$ which is also the dimension of the space $\mathcal{L}$ spanned by $\{1, \Delta, \Re((x+\mathrm{i}y)^k), \Im ((x+\mathrm{i}y)^k) : 1 \leq k \leq n-1\}$. 
It is clear that $1,\Delta \in \mathcal{H}$.
Moreover, we find
\begin{align*}
    \frac{\partial \Delta}{\partial x} & = -\frac{\mathrm{i}n}{2}  ((x+\mathrm{i}y)^{n-1}-(x-\mathrm{i}y)^{n-1})  = n\Im ((x+\mathrm{i}y)^{n-1}) \\
    \frac{\partial \Delta}{\partial y} & = \frac{n}{2} ((x+\mathrm{i}y)^{n-1}+ (x-\mathrm{i}y)^{n-1}) = n \Re ((x+\mathrm{i}y)^{n-1}).
\end{align*}
which, iteratively applied, shows that we have $\mathcal{L} \subseteq \mathcal{H}$ and thus $\mathcal{L} = \mathcal{H}$.
\end{proof}
We write $\mathcal{H}_k $ for the intersection of the homogeneous part of degree $k$ in $\RR[x,y]$ with $\mathcal{H}$. 
Since we consider a linear action, which preserves the degrees of polynomials and $\mathcal{H}_k$ is of dimension $2$, it follows that every two-dimensional irreducible representation of $I_2(n)$ can be realized as some space $\mathcal{H}_k$ for $1 \leq k \leq n-1$.
If $n$ is even we have $r \cdot (x+\mathrm{i}y)^{\frac{n}{2}} = - (x+\mathrm{i}y)^{\frac{n}{2}}$ which follows from \cref{eq: x+iy} since $\mathrm{e}^{\mathrm{i}\pi}=-1$. 
Thus, the vector spaces $\langle \Re((x+\mathrm{i}y)^{\frac{n}{2}})\rangle$ and $\langle \Im((x+\mathrm{i}y)^{\frac{n}{2}})\rangle$ are closed under the action of $I_2(n)$.
Therefore, $\mathcal{H}_{\frac{n}{2}}$ decomposes into two non-isomorphic one-dimensional irreducible representations.

\begin{lemma}
For $1 \leq k \leq \lfloor \frac{n-1}{2} \rfloor$ the representations $\mathcal{H}_k $ and $\mathcal{H}_{n-k}$ are isomorphic.
\end{lemma}
\begin{proof}
Let $p_k = \Re ((x+\mathrm{i}y)^k), q_k = -\Im ((x+\mathrm{i}y)^k)$ and $p_{n-k} = \Re ((x+\mathrm{i}y)^{n-k}), q_{n-k} = \Im ((x+\mathrm{i}y)^{n-k})$.
If $\mathrm{e}^{\mathrm{i}\frac{2\pi}{n}(n-k)} = a+ \mathrm{i} b$ is a unit root, its inverse equals $\mathrm{e}^{-\mathrm{i}\frac{2\pi}{n}(n-k)} = a- \mathrm{i} b$ and we find, using \cref{eq: x+iy},
\[r \cdot (x+\mathrm{i}y)^k = (a+ \mathrm{i} b)(x+\mathrm{i}y)^k \text{ and } r \cdot (x+\mathrm{i}y)^{n-k} = (a- \mathrm{i} b)(x+\mathrm{i}y)^{n-k} .\]
Moreover, $s \cdot p_k = p_k, s \cdot q_k = -q_k, s \cdot p_{n-k} =  p_{n-k}, s \cdot q_{n-k} =  -q_{n-k}$ and 
\begin{align*}
    r \cdot p_k & = \Re ((a+\mathrm{i}b)(x+\mathrm{i}y)^k) && = ap_k+bq_k, \\
    r \cdot p_{n-k}  & =  \Re ((a-\mathrm{i}b)(x+\mathrm{i}y)^{n-k})   && = ap_{n-k}+bq_{n-k}, \\
    r \cdot q_k & = \Im ((a+\mathrm{i}b)(x+\mathrm{i}y)^k) && = -bp_k+aq_k, \\
    r \cdot q_{n-k} & = \Im ((a-\mathrm{i}b)(x+\mathrm{i}y)^{n-k}) && = -bp_{n-k}+aq_{n-k},
\end{align*}
which shows that the map $\mathcal{H}_k \longrightarrow \mathcal{H}_{n-k}, p_k \mapsto p_{n-k}, q_k \mapsto q_{n-k}$ is $I_2(n)$-equivariant and thus an isomorphism of irreducible representations.
\end{proof}

    We denote the irreducible representations of $I_2(n)$ from now on by $V_0, V_1, \ldots, V_{\lfloor \frac{n-1}{2} \rfloor}, V_n$ and (if $n$ is even) also by $\Re V_{\frac{n}{2}}, \Im V_{\frac{n}{2}}$.

\begin{definition}
For $k \in \{0,1,\ldots,\lfloor \frac{n-1}{2} \rfloor, n\}$ and $V_k$, or $\Re V_{\frac{n}{2}}$/$\Im V_{\frac{n}{2}}$ an irreducible representation, the associated $I_2(n)$-\emph{Specht ideal} is the ideal $\mathbf{I}_k$ or $\mathbf{I}_{\frac{n}{2}}^{\Re}$/$\mathbf{I}_{\frac{n}{2}}^{\Im}$ in $\RR[x,y]$ generated by the vector space $\mathcal{H}_k$ or $\Re((x+\mathrm{i}y)^{\frac{n}{2}})$/$\Im((x+\mathrm{i}y)^{\frac{n}{2}})$. 
The associated $I_2(n)$-Specht variety is the associated variety of the Specht ideal in $\RR^2$.
\end{definition}
By a result of Steinberg \cite{steinberg1960invariants}, the polynomial $\Delta$ factors into a product of $n$ homogeneous linear polynomials whose zero sets are the reflection hyperplanes of $I_2(n)$ (see also \cite[Chapter~21-1]{kane}).
For instance, for $n=3$ we have \[\Delta = \Im((x+\mathrm{i}y)^3) = y(\sqrt{3}x-y)(\sqrt{3}x+y).\]
In particular, $V(\mathbf{I}_n)$ is the union of all reflection hyperplanes of $I_2(n)$ (see \cref{fig:reflection hyerplanes for I_2(n)} for $n=3$ and $n=8$).

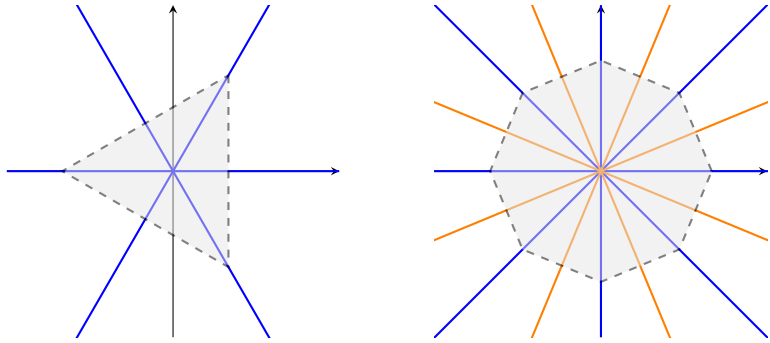
\begin{figure}[h!] 
\centering

\begin{tikzpicture}
\begin{axis}[
    axis lines=middle,
    xmin=-3, xmax=3,
    ymin=-3, ymax=3,
    width=6cm,
    height=6cm,
    ticks=none,
]
\addplot[domain=-3:3, blue, thick] {0};
\addplot[domain=-3:3, blue, thick] {sqrt(3)*x};
\addplot[domain=-3:3, blue, thick] {-sqrt(3)*x};

\addplot[draw=black, dashed, thick, fill=gray!20, opacity=0.5] coordinates {
    (1,{sqrt(3)}) (1,-{sqrt(3)}) (-2,0) (1,{sqrt(3)})
};
\end{axis}
\end{tikzpicture}
\hspace{1cm}
\begin{tikzpicture}
\begin{axis}[
    axis lines=middle,
    xmin=-3, xmax=3,
    ymin=-3, ymax=3,
    width=6cm,
    height=6cm,
    ticks=none,
]
\addplot[domain=-4:4, blue, thick] {0}; 
\addplot[domain=-4:4, blue, thick] coordinates {(0,-3)(0,3)}; 

\addplot[domain=-4:4, blue, thick] {x}; 
\addplot[domain=-4:4, blue, thick] {-x}; 
\addplot[domain=-4:4, orange, thick] {(1 + sqrt(2))*x};
\addplot[domain=-4:4, orange, thick] {(sqrt(2) - 1)*x};
\addplot[domain=-4:4, orange, thick] {(1 - sqrt(2))*x};
\addplot[domain=-4:4, orange, thick] {-(1 + sqrt(2))*x};

\addplot[draw=black, dashed, thick, fill=gray!20, opacity=0.5] coordinates {
    (2,0) 
    (1.4142, 1.4142) 
    (0,2) 
    (-1.4142, 1.4142) 
    (-2,0) 
    (-1.4142, -1.4142) 
    (0, -2) 
    (1.4142, -1.4142) 
    (2,0)
};
\end{axis}
\end{tikzpicture}

\caption{The reflection hyperplanes for $I_2(3)$ and $I_2(8)$ together with corresponding regular $n$-gons}
\label{fig:reflection hyerplanes for I_2(n)}
\end{figure}
Observe that for even $n = 2m$ the polynomial $\Delta$ factors into $\Delta = 2\Re (x+\mathrm{i}y)^m \cdot \Im (x+\mathrm{i}y)^m$ since $(a+\mathrm{i}b)^2 = a^2-b^2 + 2\mathrm{i}ab$ for real numbers $a,b$. 
This implies that the variety of $\Re ((x+\mathrm{i}y)^k)$ is the union of the reflection hyperplanes of $I_2(2k)$ that are not reflection hyperplanes of $I_2(k)$.
The reflection hyperplanes of $I_2(8)$ which are not reflection hyperplanes of $I_2(4)$ are colored orange in \cref{fig:reflection hyerplanes for I_2(n)}. \\
Thus, we immediately find that the variety of the ideal $\mathbf{I}_k$ equals $\{(0,0)\}$ for $1 \leq k \leq \lfloor \frac{n-1}{2} \rfloor$, since the reflection hyperplanes of $I_2(n)$ intersect only in the origin.

\begin{theorem}\label{thm: main1dihedral}
The Specht variety of the dihedral group $I_2(n)$ is one of the following:
\begin{enumerate}
    \item $\emptyset$ for $V_0$ the trivial representation.
    \item $\{(0,0)\}$ for $V_k$ and $1 \leq k \leq \lfloor \frac{n-1}{2} \rfloor$.
    \item The union of all reflection hyperplanes of $I_2(\frac{n}{2})$ for $\Im V_{\frac{n}{2}}$.
    \item The union of all reflection hyperplanes of $I_2(n)$ which are not reflection hyperplanes of $I_2(\frac{n}{2})$ for $\Re V_{\frac{n}{2}}$.
    \item The union of all reflection hyperplanes of $I_2(n)$ for $V_n$.  
\end{enumerate}
\end{theorem}
Moreover, the Specht ideals of the dihedral group are totally ordered with respect to inclusion for odd $n$ and otherwise ``almost'' totally ordered.
\begin{theorem}\label{thm:main2dihedral}
The Specht ideals of the dihedral group $I_2(n)$ are ordered as follows:
\[
    \mathbf{I}_0 \supsetneq \mathbf{I}_1 \supsetneq \ldots \supsetneq \mathbf{I}_{\lfloor \frac{n-1}{2 } \rfloor} \supsetneq  \mathbf{I}_{\frac{n}{2}}^{\Re} / \mathbf{I}_{\frac{n}{2}}^{\Im} \supsetneq \mathbf{I}_n
\]
where $\mathbf{I}_{\frac{n}{2}}^{\Re} $ and $\mathbf{I}_{\frac{n}{2}}^{\Im}$ are incomparable but both contain $\mathbf{I}_n$. 
Moreover, the isotypic component of any irreducible representation in $\RR[x,y]$ is contained in the associated Specht ideal.
The only radical Specht ideals are $\mathbf{I}_0, \mathbf{I}_1, \mathbf{I}_n$ and $\mathbf{I}_{\frac{n}{2}}^{\Re}, \mathbf{I}_{\frac{n}{2}}^{\Im}$ if $n$ is even.
\end{theorem}

\begin{proof}
\begin{enumerate}
    \item First, we prove that $\mathbf{I}_k \supseteq \mathbf{I}_{k+1}$ for $0 \leq k \leq \lfloor \frac{n-1}{2} \rfloor$.

    Consider $(x+\mathrm{i}y)^k=\sum_{j=0}^k{k\choose j}x^{k-j}y^j\mathrm{i}^j$.
Consequently, we get
\begin{align*}
    \Re((x+\mathrm{i}y)^k)&=\sum_{j=0}^{\lceil\frac{k-1}{2}\rceil}{k\choose 2j}x^{k-2j}y^{2j}(-1)^j,\\
    \Im((x+iy)^k)&=\sum_{j=0}^{\lfloor\frac{k-1}{2}\rfloor}{k\choose 2j+1}x^{k-2j-1}y^{2j+1}(-1)^j.
\end{align*}
It follows for odd $k$
\begin{align*}
    \Re((x+\mathrm{i}y)^k)\cdot x-\Im((x+\mathrm{i}y)^k)\cdot y = &\sum_{j=0}^{\frac{k-1}{2}}{k\choose 2j}x^{k-2j+1}y^{2j}(-1)^j-\sum_{j=0}^{\frac{k-1}{2}}{k\choose 2j+1}x^{k-2j-1}y^{2j+2}(-1)^j\\
    = &\sum_{j=0}^{\frac{k-1}{2}}{k\choose 2j}x^{k-2j+1}y^{2j}(-1)^j-\sum_{j=1}^{\frac{k-1}{2}}{k\choose 2j-1}x^{k-2j+1}y^{2j}(-1)^{j-1}\\&-{k\choose 2\frac{k-1}{2}+1}x^{k-2(\frac{k-1}{2}+1)}y^{2\frac{k-1}{2}+2}(-1)^{\frac{k-1}{2}+1}\\
    = &x^{k+1}+\sum_{j=1}^{\frac{k-1}{2}}\left({k\choose 2j}+{k\choose 2j-1}\right)x^{k-2j+1}y^{2j}(-1)^{j}-y^{k+1}(-1)^{\frac{k-1}{2}+1}\\
    = & x^{k+1}+\sum_{j=1}^{\frac{k-1}{2}}{k+1\choose 2j}x^{k-2j+1}y^{2j}(-1)^{j}+y^{k+1}(-1)^{\frac{k+1}{2}}\\
    = &\sum_{j=0}^{\frac{k+1}{2}}{k+1\choose 2j}x^{k-2j+1}y^{2j}(-1)^{j} \\ 
    = & \Re((x+\mathrm{i}y)^{k+1}).
\end{align*}
Analogously, we have

\begin{align*}
    \Re((x+\mathrm{i}y)^k)\cdot y+\Im((x+\mathrm{i}y)^k)\cdot x = &\sum_{j=0}^{\frac{k-1}{2}}{k\choose 2j}x^{k-2j}y^{2j+1}(-1)^j+\sum_{j=0}^{\frac{k-1}{2}}{k\choose 2j+1}x^{k-2j}y^{2j+1}(-1)^j\\
    = &\sum_{j=0}^{\frac{k-1}{2}}\left( {k\choose 2j} + {k\choose 2j+1}\right) x^{k-2j}y^{2j+1}(-1)^j\\
    = &\sum_{j=0}^{\frac{k-1}{2}} {k+1\choose 2j+1}x^{k-2j}y^{2j+1}(-1)^j \\ = & \Im (x+\mathrm{i}y)^{k+1}.
\end{align*}
The same procedure works similarly for even $k$.\\
Thus, $\Re((x+\mathrm{i}y)^{k+1}), \Im((x+\mathrm{i}y)^{k+1}) \in (V_k)\subseteq \RR[x,y]$ proving the claim about the ideal inclusion together with the identity $\Im((x+\mathrm{i}y)^{2k})=2\Re((x+\mathrm{i}y)^{k})\cdot \Im((x+\mathrm{i}y)^{k})$.

\item The claim on the isotypic component follows from the observation that for the two-dimensional irreducible representations $V_k \cong V_{n-k}$ in the space of harmonics $\mathcal{H}$ we have $\mathcal{H}_{n-k} \subseteq I_k$ by the proof in point (1).

\item The ideals $\mathbf{I}_0 = (1), \mathbf{I}_1 = (x,y)$ are clearly radical. Since $\Im ((x+\mathrm{i}y)^k)$ and $\Re ((x+\mathrm{i}y)^k)$ are products of pairwise linearly independent homogeneous linear polynomials, the radicality also follows for $\mathbf{I}_n$ and $\mathbf{I}_{ \frac{n}{2} }^{\Re },  \mathbf{I}_{ \frac{n}{2} }^{\Im}$.
All other Specht ideals cannot be radical, since their varieties equal $\{(0,0)\}$ but the ideals are not $(x,y)=\mathbf{I}_1$.
\end{enumerate}

\end{proof}

We can formulate a direct consequence on the solution of low degree polynomial systems of equations that are closed under the action of dihedral groups.

\begin{remark}
Let $I \subseteq \RR[x,y]$ be a $I_2(n)$-invariant ideal. If $I \cap \RR[x,y]_{< \frac{n}{2}}$ is non-empty then the variety of $I$ is a subset of $\{(0,0)\}$.
\end{remark}

\section{Preliminaries for type $D$}\label{sec:pre. for D_n case}
Throughout the remaining sections, let $\KK$ be a field of characteristic zero and $n \in \NN$.
We write $\KK[\underline{X}]$ for the polynomial ring $\KK[X_1,\ldots,X_n]$.
Recall that the \emph{symmetric group} $S_n$ is the group of all permutations of $[n]$ which acts on $\KK[\underline{X}]$ by permuting the variables. 
The \emph{hyperoctahedral group} $B_n$ is the wreath product of the cyclic group $C_2$ and the symmetric group $S_n$, i.e. \[B_n \cong C_2 \wr S_n = \{((s_1,\ldots,s_n),\sigma) : s_i \in \{\pm 1\}, \sigma \in S_n\}.\]
The group $B_n$ is also known as the signed symmetric group. 
$B_n$ acts on $\KK[\underline{X}]$ by permuting the variables and switching the signs $X_i \mapsto -X_i$.\\
The \emph{even signed symmetric group} $D_n$ is the subgroup of $B_n$ of index $2$ which always switches an even number of signs. 
More precisely, \[D_n = \{((s_1,\ldots,s_n),\sigma) : s_i \in \{\pm 1\}, \sigma \in S_n, \prod_{i=1}^n s_i = 1\}\] and $D_n$ acts on $\KK[\underline{X}]$ via the restriction of the action of $B_n$.
For background on the groups $S_n, B_n$ and $D_n$ and their representations on $\KK[\underline{X}]$ we refer to \cite{musili} and introduce the relevant notation to study Specht ideals below. 
We briefly motivate the connection of Specht ideals to representation theory in \cref{rem:Connection to repr theory}.

\subsection{Partitions and tableaux}

A \emph{partition} $\lambda$ of $n$ is a sequence of non-increasing positive integers $(\lambda_1,\lambda_2,\ldots,\lambda_l)$ for which $\sum_{i=1}^l\lambda_i=n$.
We write $\lambda \vdash n$ if $\lambda$ is a partition of $n$. 
We also allow to add an arbitrary number of zeros after $\lambda_l$ and identify the obtained sequence with the partition.
The maximal $l$ for which $\lambda_l \neq 0$ is called the length of $\lambda$ and is denoted by $\operatorname{len}(\lambda)$. 
Although, strictly speaking, there is no partition of $0$ according to the definition above, we say that $\emptyset$ is the unique partition of $0$.
A \emph{bipartition} of $n$ is a pair of partitions $(\lambda,\mu)$ with $\lambda \vdash k$ and $\mu \vdash n-k$ for some integer $0 \leq k \leq n$. 
A \emph{dipartition} of $n$ is an unordered pair of distinct partitions $\{\lambda,\mu\}$ which are partitions of $k$ and $n-k$ for some $k$, or a set of the form $\{\lambda,+\}$ or $\{\lambda,-\}$ where $\lambda \vdash \frac{n}{2}$ is a partition (which requires $n$ to be even). 
We often write $\{\lambda,\pm\}$ and mean that the dipartition is either $\{\lambda,+\}$ or $\{\lambda,-\}$.
For two partitions $\lambda \vdash k,\mu \vdash n-k$ we call the unique partition of $n$, which is the reordering of the concatenation of $\lambda$ and $\mu$, the \emph{fusion} of $\lambda$ and $\mu$ and denote it by $\lambda \uplus \mu$.\\
We denote the set of partitions/bipartitions/dipartitions of $n$ by $\mathcal{P}_n$/$\mathcal{B}_n$/$\mathcal{D}_n$ and we equip the first two with known partial orders.
The \emph{dominance order} $\unlhd_S$ (see e.g. \cite{brylawski1973lattice} for a study of the lattice $(\mathcal{P}_n,\unlhd_S)$) is a partial order on the set $\mathcal{P}_n$ defined by $\mu \unlhd_S \lambda$, if 
\begin{align*}
\sum_{i=1}^j \mu_i \leq \sum_{i=1}^j \lambda_i, \text{ for all integers } 1 \leq j \leq n.
\end{align*}
The \emph{bidominance order} introduced in \cite{debus_moustrou_riener_verdure} is the partial order $\unlhd_B$ on the set $\mathcal{B}_n$ defined by $(\lambda,\mu) \unlhd_B (\omega,\theta)$ if 
\begin{align*}
\sum_{i=1}^j (\lambda_i+ \mu_i) \leq \sum_{i=1}^j (\omega_i + \theta_i) \text{ and } \lambda_{j}+\sum_{i=1}^{j-1} (\lambda_i+ \mu_i) \leq \omega_{j} + \sum_{i=1}^{j-1} (\omega_i + \theta_i), \text{ for all integers } 1 \leq j \leq n.    
\end{align*}
Note that the first chain of inequalities is equivalent to the dominance relation of the partitions $\lambda \uplus \mu$ and $\omega \uplus \theta$.
We introduce a partial order on $\mathcal{D}_n$ in Section \ref{sec:Partial order}.\\
The \emph{diagram} of a partition $\lambda$ is the ordered collection of boxes, arranged in left-justified, top-aligned rows where the $i$-th row contains $\lambda_i$ many boxes for all positive integers $i$.
We frequently identify partitions with their associated diagrams.
A filling of a diagram of shape $\lambda \vdash n$ using all of the integers in $[n]$ is called a \emph{tableau}. 
We denote the set of all tableaux of shape $\lambda$ by $\YT(\lambda)$. \\
For a sequence $K=(k_1,\ldots,k_\ell)$ of pairwise distinct integers in $[n]$ let 
\[
    \Delta_K(\underline{X})\coloneqq \prod_{1 \leq i < j \leq \ell}(X_{k_i}-X_{k_j}) \in \ZZ[\underline{X}] \subseteq \KK[\underline{X}]
\] 
its associated \emph{Vandermonde determinant}, where $\Delta_\emptyset(\underline{X})\coloneqq 1$. 
If $T$ is a tableau with $m$ columns and $C_i$ denotes the sequence of entries of the $i$-th column of $T$ written from top to bottom, then \[\spe_T^\mathbf{S}(\underline{X}) \coloneqq  \Delta_{C_1}(\underline{X})\cdots \Delta_{C_m}(\underline{X})\] is the $\mathbf{S}$-\emph{Specht polynomial} of $T$. The index $\mathbf{S}$ in $\spe_T^\mathbf{S}$ refers to the symmetric group $S_n$.

\begin{example}
For instance, $\ytableausetup{smalltableaux}
\begin{ytableau}
\empty &  &  & \\
 \empty &  & \\ 
\empty &
\end{ytableau}$
is a diagram of shape $\mu=(4,3,2)$ and $ S =
\begin{ytableau}
3 & 2 & 8 & 6\\
 1 & 4 &5 \\ 
9 & 7
\end{ytableau}$
is a tableau of shape $\mu$. Here, $C_1 = (3,1,9)$ is the sequence corresponding to the first column of $S$. Its associated Vandermonde determinant is $\Delta_{C_1}(\underline{X})=(X_3-X_1)(X_3-X_9)(X_1-X_9)$ and \[\spe_S^\mathbf{S}(\underline{X})=(X_3-X_1)(X_3-X_9)(X_1-X_9)(X_2-X_4)(X_2-X_7)(X_4-X_7)(X_8-X_5).\]
\end{example}

We continue with analogous definitions for bipartitions and dipartitions. 
The \emph{bidiagram} of a bipartition $(\lambda,\mu)$ is the pair of diagrams of shape $\lambda$ and $\mu$. 
A \emph{bitableau} of shape $(\lambda,\mu)$ is a filling of the bidiagram with the integers $[n]$ such that any integer occurs precisely once. 
We denote the set of all bitableaux of shape $(\lambda,\mu)$ by $\YT(\lambda,\mu)$.
The \emph{$\mathbf{B}$-Specht polynomial} $\spe_{(T,S)}^\mathbf{B}$ associated with a bitableau $(T,S)$ is the polynomial 
\[\spe_{(T,S)}^\mathbf{B}(\underline{X})\coloneqq \spe_T^\mathbf{S}(\underline{X}^2)\spe_S^\mathbf{S}(\underline{X}^2)\prod_{j \in S}X_j,\]
where $\underline{X}^2 \coloneqq  (X_1^2,\ldots,X_n^2)$.
The index $\mathbf{B}$ in $\spe_{(T,S)}^\mathbf{B}$ refers to the hyperoctahedral group $B_n$.

\begin{example}
The bidiagram $\left(
\ytableausetup{smalltableaux,centertableaux}
\begin{array}{cc}
\begin{ytableau}
\empty & \empty  \\
\empty & \empty  \\
\empty & \empty 
\end{ytableau}
~,~&
\begin{ytableau}
\empty & \empty  \\
\empty & \none \\
\none
\end{ytableau}
\end{array}
\right)$ is of shape $(\lambda,\mu)=((2,2,2),(2,1))$ and $(T,S)=\left(
\ytableausetup{smalltableaux,centertableaux}
\begin{array}{cc}
\begin{ytableau}
3 & 2  \\
1& 4  \\
9 & 7 
\end{ytableau}
~,~&
\begin{ytableau}
8 & 6  \\
5 & \none \\
\none
\end{ytableau}
\end{array}\right)$
is a bitableau of shape $(\lambda,\mu)$. 
The corresponding $\mathbf{B}$-Specht polynomial  is \[\spe_{(T,S)}^\mathbf{B}=(X_3^2-X_1^2)(X_3^2-X_9^2)(X_1^2-X_9^2)(X_2^2-X_4^2)(X_2^2-X_7^2)(X_4^2-X_7^2)(X_8^2-X_5^2)X_5X_6X_8.\] 
By permuting the two tableaux, we obtain the pair $(S,T)$ which is a bitableau of shape $(\mu,\lambda)$ and $\spe_{(S,T)}^\mathbf{B}$ equals
\[(X_3^2-X_1^2)(X_3^2-X_9^2)(X_1^2-X_9^2)(X_2^2-X_4^2)(X_2^2-X_7^2)(X_4^2-X_7^2)(X_8^2-X_5^2)X_1X_2X_3X_4X_7X_9.\]
\end{example}

For a dipartition of the form $\{\lambda,\pm\}$ and a bitableau $(T,S)$ of shape $(\lambda,\lambda)$ we define the \emph{$\mathbf{D}$-Specht polynomials} 
\begin{align*}
 \spe_{\{(T,S),+\}}^\mathbf{D}(\underline{X})\coloneqq & \spe_{(T,S)}^\mathbf{B}(\underline{X})+\spe_{(S,T)}^\mathbf{B}(\underline{X})=\spe_T^\mathbf{S}(\underline{X}^2)\spe_S^\mathbf{S}(\underline{X}^2)(\prod_{i \in T}X_i+\prod_{j \in S}X_j), \\
 \spe_{\{(T,S),-\}}^\mathbf{D}(\underline{X})\coloneqq  & \spe_{(T,S)}^\mathbf{B}(\underline{X})-\spe_{(S,T)}^\mathbf{B}(\underline{X})=\spe_T^\mathbf{S}(\underline{X}^2)\spe_S^\mathbf{S}(\underline{X}^2)(\prod_{i \in T}X_i-\prod_{j \in S}X_j).
\end{align*}
The index $\mathbf{D}$ in $\spe_{(T,S)}^\mathbf{D}$ refers to the even signed symmetric group $D_n$.

\begin{example}
Let $\{\lambda,-\}$ be the dipartition of $10$ with $\lambda = (3,2)$. Then for the bitableau $(T,S)=\left(
\ytableausetup{smalltableaux,centertableaux}
\begin{array}{cc}
\begin{ytableau}
3 & 2 & 10 \\
9& 5  
\end{ytableau}
~,~&
\begin{ytableau}
4 & 1 & 8 \\
7 & 6
\end{ytableau}
\end{array}\right)$ of shape $(\lambda,\lambda)$ 
 its associated $\mathbf{D}$-Specht polynomial is
\[\spe_{(T,S)}^\mathbf{D} (\underline{X})= (X_3^2-X_9^2)(X_2^2-X_5^2)(X_4^2-X_7^2)(X_1^2-X_6^2)(X_2X_3X_5X_9X_{10}+X_1X_4X_6X_7X_8).\]
\end{example}
\subsection{Specht ideals}
We recall the definitions of Specht ideals which were studied for $S_n$ and $B_n$ and introduce them for $D_n$.
\begin{definition}
Let $\lambda \in \mathcal{P}_n, (\lambda,\mu) \in \mathcal{B}_n$ and $\Lambda \in \mathcal{D}_n$.
\begin{itemize}
    \item The $\mathbf{S}$-\emph{Specht ideal} of $\lambda$, denoted $I_\lambda^{\mathbf{S}}$, is the ideal in $\KK[\underline{X}]$ generated by all Specht polynomials indexed by tableaux of shape $\lambda$, i.e. $I_\lambda^{\mathbf{S}} \coloneqq  (\spe_T^\mathbf{S} \, \, \mid \, \, T \in \YT(\lambda))$. 
The \emph{Specht variety} $V_\lambda^{\mathbf{S}}$ is the algebraic variety of $I_\lambda^{\mathbf{S}} $ over the field $\KK$, i.e. $V_\lambda^{\mathbf{S}} = \{ x \in \KK^n \, \, : \, \, \spe_T^\mathbf{S}(x)=0,\, \forall T \in \YT(\lambda)\}$.
\item The $\mathbf{B}$-\emph{Specht ideal} of $(\lambda,\mu)$, denoted $I_{(\lambda,\mu)}^{\mathbf{B}}$, is the ideal generated by all $\mathbf{B}$-Specht polynomials indexed by bitableaux of shape $(\lambda,\mu)$.
The \emph{Specht variety} $V_{(\lambda,\mu)}^{\mathbf{B}}$ is the algebraic variety of $I_{(\lambda,\mu)}^{\mathbf{B}} $ over the field $\KK$.
\item The $\mathbf{D}$-\emph{Specht ideal} of $\Lambda$ is denoted by $I_{\Lambda}^{\textbf{D}}$. 
If $\Lambda = \{\lambda,\pm\}$ then $I_{\Lambda}^\mathbf{D}$ is the ideal generated by all $\spe_{\{(T,S),\pm\}}^{\mathbf{D}}(\underline{X})$ for $(T,S) \in \YT(\lambda,\lambda)$. 
Otherwise, for $\Lambda = \{\lambda,\mu\}$ the ideal $I_{\Lambda}^{\mathbf{D}}$ is generated by all $\mathbf{B}$-Specht polynomials whose bitableaux are of shape $(\lambda,\mu)$ or $(\mu,\lambda)$, i.e. $I_\Lambda^\mathbf{D} = I_{(\lambda,\mu)}^\mathbf{B} + I_{(\mu,\lambda)}^\mathbf{B}$. The $\mathbf{D}$-Specht variety $V_\Lambda^\mathbf{D}$ is the algebraic variety of $I_\Lambda^\mathbf{D}$ over the field $\KK$.
\end{itemize}
\end{definition}

For a group $G$ acting on $\KK[\underline{X}]$ an ideal $I\subseteq \KK[\underline{X}]$ is $G$-\emph{invariant} if $g \cdot f \in I$ for all $f \in I$ and $g \in G$.
Each Specht ideal is invariant with respect to the action of its corresponding group by construction.
In particular, the $\mathbf{B}$-Specht ideals are also $D_n$ and $S_n$-invariant, and $\mathbf{D}$-Specht ideals are also $S_n$-invariant. 
However, the $\mathbf{D}$-Specht ideals are in general not $B_n$-invariant, e.g. $I_{\{(2,2),+\}}^\mathbf{D}$ is not $B_8$-invariant.

\begin{remark}\label{rem:Connection to repr theory}
 The definition of Specht ideals is motivated by the representation theory of the underlying group. For $S_n$ and $B_n$ and $\operatorname{char}(\KK)=0$, the Specht polynomials were introduced by Wilhelm Specht \cite{ specht1937Bdarstellungstheorie, specht1937Adarstellungstheorie}. He proved that the vector spaces spanned by the Specht polynomials of a fixed shape, called \emph{Specht modules}, are precisely the pairwise non-isomorphic, irreducible representations of the groups $S_n$ and $B_n$. 
 The Specht modules are often denoted by $\mathbb{S}^\lambda$/$\mathbb{S}^{(\lambda,\mu)}$.
 This construction was extended to the subgroup $D_n$ of $B_n$ of index $2$ using Clifford theory (see e.g. \cite[Chapter~8]{musili} or \cite{morita1998higher}). 
 For $\lambda \neq \mu$ every $B_n$-irreducible representation $\mathbb{S}^{(\lambda,\mu)}$ is also $D_n$-irreducible and  $\mathbb{S}^{(\lambda,\mu)}$, $\mathbb{S}^{(\mu,\lambda)}$ are $D_n$-isomorphic (\emph{loc. cit.}) and we write $\mathbb{S}^{\{\lambda,\mu\}}$ instead. Moreover, every $B_n$-irreducible representation $\mathbb{S}^{(\lambda,\lambda)}$ decomposes into two non-isomorphic $D_n$-irreducible representations 
 \begin{align*}
 \mathbb{S}^{\{\lambda,+\}} \coloneqq  &\langle \spe_{(T,S)}(\underline{X})+\spe_{(S,T)}(\underline{X}) : (T,S) \in \YT(\lambda,\lambda) \rangle, \\ 
 \mathbb{S}^{\{\lambda,-\}} \coloneqq & \langle \spe_{(T,S)}(\underline{X})-\spe_{(S,T)}(\underline{X}) : (T,S) \in \YT(\lambda,\lambda)\rangle.      \end{align*}
These are all $D_n$-irreducible representations (\emph{loc. cit.}). 
\end{remark} 

In fact, if $\operatorname{char}(\mathbb{K})=0$ then for $S_n$ and $B_n$ an irreducible representation isomorphic to a Specht module in $\KK[\underline{X}]$ is contained in the corresponding Specht ideal \cite{debus_moustrou_riener_verdure,moustrou2021symmetric, woo2005ideals}.
In other words, the isotypic component of $\KK[\underline{X}]$ with respect to $\mathbb{S}^\lambda$/$\mathbb{S}^{(\lambda,\mu)}$ is contained in $I_\lambda^{S_n}$/$I_{(\lambda,\mu)}^{B_n}$.  
Actually, an even stronger property holds. Namely, if $\spe_T \mapsto f$ defines an equivariant isomorphism of $S_n$-modules then $f$ is divisible by $\spe_T$ in $\KK[\underline{X}]$ \cite{moustrou2021symmetric, woo2005ideals}. 
The analogous statement holds for $B_n$ \cite[Proposition~7.6]{debus_moustrou_riener_verdure}.
The following example shows that this property does not hold for the group $D_n$.

\begin{example}\label{example:dividing specht polynomials}
For $(\lambda,\lambda)=((1,1),(1,1))$ each of the $B_4$-orbits of the polynomials $f_1(\underline{X}) = (X_1^4-X_2^4)(X_3^2-X_4^2)X_3X_4$ and $f_2(\underline{X}) = (X_1^2-X_2^2)(X_3^4-X_4^4)X_1X_2$ spans an irreducible representation $\mathbb{S}^{(\lambda,\lambda)}$ by \cite[Section~5]{morita1998higher}. 
Moreover, the map $g(\underline{X})\coloneqq (X_1^2-X_2^2)(X_3^2-X_4^2)(X_1X_2+X_3X_4) \mapsto f_1+f_2$ induces a $D_4$-equivariant isomorphism (\emph{loc. cit.}). 
However, the $\mathbf{D}$-Specht polynomial $g$ does not divide $f_1+f_2$.
Nevertheless, we find $f_1+f_2 \in I_{\{\lambda,+\}}^\mathbf{D}$ with SAGE \cite{sagemath} which shows that containment of the isotypic component in the Specht ideal is not ruled out by this example.
\end{example}
We can at least say something about $f$ if $\spe_{(T,S),+}^\mathbf{D} \mapsto f \in \KK[\underline{X}]$ is a $D_n$-equivariant isomorphism. 
\begin{observation} \label{observation1}
If $\spe_{(T,S),+}^\mathbf{D} \mapsto f \in \KK[\underline{X}]$ is a $D_n$-equivariant isomorphism then $f$ must be of the form 
\begin{align}\label{equation form f}
 f(\underline{X}) = \spe_T^{\mathbf{S}}(\underline{X}^2)\spe_S^{\mathbf{S}}(\underline{X}^2) \left( R(\underline{X}^2) \cdot \prod_{i \in T}X_i +  \tau(R(\underline{X}^2)) \cdot \prod_{i \in S}X_i\right),    
\end{align} 
where $R \in \KK[\underline{X}]$ and $\tau \in S_n$ is an involution which, when applied entrywise, maps $T$ to $S$.
The same holds for equivariant isomorphisms $\spe_{(T,S),-}^\mathbf{D} \mapsto f \in \KK[\underline{X}]$.
\end{observation}
A proof of \cref{observation1} can be found in \cref{proof of Observation1}.
In \cref{example:dividing specht polynomials} we have $R(\underline{X})=X_1-X_2$. 
For dipartitions of the form $\{\lambda,\mu\}$ the containment of the isotypic component in the Specht ideal follows from the $B_n$-case.

\begin{corollary}
\label{cor:isotypic decomposition}
For $\lambda \neq \mu$ the isotypic component of $\KK[\underline{X}]$ with respect to $\mathbb{S}^{\{\lambda,\mu\}}$ is contained in $I_{\{\lambda,\mu\}}^\mathbf{D}$.   
\end{corollary}
\begin{proof}
Let $\mathbb{K}[\underline{X}]_{\{\lambda,\mu\}}$ denote the $D_n$-isotypic component of $\mathbb{K}[\underline{X}]$ with respect to the dipartion ${\{\lambda,\mu\}}$ and let $\mathbb{K}[\underline{X}]_{(\lambda,\mu)}, \KK[\underline{X}]_{(\mu,\lambda)}$ denote the associated $B_n$-isotypic components of $\KK[\underline{X}]$.
Then $\KK[\underline{X}]_{\{\lambda,\mu\}} = \KK[\underline{X}]_{(\lambda,\mu)}\oplus\KK[\underline{X}]_{(\mu,\lambda)}\subseteq I_{(\lambda,\mu)}^{\mathbf{B}} + I_{(\mu,\lambda)}^{\mathbf{B}} = I_{\{\lambda,\mu\}}^\mathbf{D}$ by \cref{rem:Connection to repr theory}.
\end{proof}

\begin{remark}
It is an open question whether we have $\KK[\underline{X}]_{\{\lambda,\pm\}} \subseteq I_{\{\lambda,\pm\}}^\mathbf{D}$. 
Maybe \cref{equation form f} can help in establishing a positive answer.
\end{remark}

\subsection{Orbit types and Specht varieties}

The Specht varieties are unions of intersections of reflection hyperplanes of the underlying reflection group (i.e. of the braid arrangements) and are closed under the action of the group. 
Thus, whether a point lies in the Specht variety does depend solely on the orbit type of the point (defined in \cref{def:orbit types}). 
Moreover, an interesting combinatorial result on $\mathbf{S}$- and $\mathbf{B}$-Specht varieties is that these sets can be expressed as disjoint unions of orbit sets using the partial orders $\unlhd_S$ and $\unlhd_B$ (see \cref{thm:SnBnEquivalence}).

\begin{figure}[!htb]
    \centering
    \begin{minipage}{.5\textwidth}
        \centering

\tdplotsetmaincoords{80}{110}
\begin{tikzpicture}[tdplot_main_coords, scale=2.65]

\draw[->] (0,0,0) -- (1.5,0,0) node[anchor=north east]{$x_1$};
\draw[->] (0,0,0) -- (0,1.3,0) node[anchor=north west]{$x_2$};
\draw[->] (0,0,0) -- (0,0,1) node[anchor=south]{$x_3$};

\draw[blue, thick] (-1,-1,-1) -- (1,1,1) node[anchor=west]{$x_1 = x_2 = x_3$};

\draw[red, thick] (-0.7,0.7,-0.7) -- (0.7,-0.7,0.7) node[anchor=south]{$x_1 = -x_2 = x_3$};

\draw[cyan, thick] (-0.7,-0.7,0.7) -- (0.7,0.7,-0.7) node[anchor=west]{$x_1 = x_2 = -x_3$};

\end{tikzpicture}
    \caption{The Specht variety $V_{\{(2,1),\emptyset\}}^\mathbf{D} = V_{(\emptyset,(2,1))}^\mathbf{B} \subseteq \RR^3 $}
    \label{fig:D-Specht variety}
    \end{minipage}%
    \begin{minipage}{0.5\textwidth}
        \centering
            \begin{tikzpicture}
  \begin{axis}[
    view={120}{30},
    axis lines=center,
    xlabel={$x_1$}, ylabel={$x_2$}, zlabel={$x_3$},
    xmin=-2, xmax=2,
    ymin=-2, ymax=2,
    zmin=-2, zmax=2,
    ticks=none,
    colormap/cool,
    ]

    \addplot3[
      surf,
      opacity=0.4,
      domain=-2:2,
      y domain=-2:2,
    ]
    ({x}, {x}, {y});

    \addplot3[
      surf,
      opacity=0.4,
      domain=-2:2,
      y domain=-2:2,
    ]
    ({x}, {y}, {x});

    \addplot3[
      surf,
      opacity=0.4,
      domain=-2:2,
      y domain=-2:2,
    ]
    ({x}, {y}, {y});

  \end{axis}
\end{tikzpicture}
\caption{The Specht variety $V_{(1,1,1)}^\mathbf{S}\subseteq \RR^3$}
        \label{fig:S-Specht variety}
    \end{minipage}
\end{figure}

\begin{definition}\label{def:orbit types}
\begin{itemize}
    \item For a partition $\lambda = (\lambda_1,\ldots,\lambda_l) \in \mathcal{P}_n$ the \emph{orbit set} $O(\lambda)$ of type $\lambda$ is the set of all points in $\KK^n$ whose stabilizer subgroup of $S_n$ is isomorphic to $S_\lambda \coloneqq  S_{\lambda_1}\times \ldots \times S_{\lambda_l}$. Conversely, we say that a point $x \in O(\lambda)$ has \emph{orbit type} $\lambda$. 
    \item For a bipartition $(\lambda,\mu) \in \mathcal{B}_n$ with the property that $\lambda_{i+1} = \lambda_1$ whenever $\mu_i > 0$, we define the \emph{orbit set} $O(\lambda,\mu)$ as the $B_n$-orbits of all points of the form \[(\underbrace{0,\ldots,0}_{\lambda_1},\underbrace{\pm a_1,\ldots,\pm a_1}_{\mu_1+\lambda_2},\ldots,\underbrace{\pm a_l,\ldots,\pm a_l}_{\mu_l+\lambda_{l+1}})\]
    with $0 \neq a_i^2  \neq a_j^2 $ for all $1 \leq i \neq j \leq l$.
    For every other bipartition $(\lambda,\mu)$ we set $O(\lambda,\mu)\coloneqq \emptyset$. We say that a point $x \in O(\lambda,\mu)$ has \emph{orbit type} $(\lambda,\mu)$.
\end{itemize}
From the definition it is evident that orbit sets are pairwise disjoint. 
In particular, the definitions of orbit sets for the symmetric group and the hyperoctahedral group induce set partitions of $\KK^n$ labeled by partitions and bipartitions, respectively.
Sometimes, we distinguish more precisely between $\mathbf{S}$-orbit sets $O(\lambda)$ and $\mathbf{S}$-orbit types, and $\mathbf{B}$-orbit sets $O(\lambda,\mu)$ and $\mathbf{B}$-orbit types.
\end{definition}
\begin{example}
\begin{itemize}
    \item $O(2,2,1)=\{ \sigma \cdot (a,a,b,b,c) \, \, : \, \, a,b,c \in \KK, a \neq b \neq c \neq a, \sigma \in S_5\}$.
    \item The bipartitions $(\lambda,\mu)$ that have non-empty $\mathbf{B}$-orbit set are those arising as a cut of a partition of $n$ \cite[Definition~6.2]{debus_moustrou_riener_verdure}. More precisely, if $x \in \KK^n$ is such that $x^2=(x_1^2,\ldots,x_n^2)$ has $\mathbf{S}$-orbit type $\kappa \vdash n$ and $t$ is the number of $0$ coordinates of $x$ then the $\mathbf{B}$-orbit type of $x$ is $((t,\ldots,t,\lambda_{s+1},\ldots,\lambda_l),(\lambda_1-t,\ldots,\lambda_s-t))$ where $s$ is the maximal integer with $\lambda_s = t$.
    \item $O((2,2),(1,1))=\emptyset$ since $\len (1,1)=2$ but $(2,2)_{2+1}=0<2= (2,2)_1$, and $O((2,2),(1))=\{\sigma \cdot (0,0,a,a,a) \, \, : \, \, 0 \neq a \in \KK, \sigma \in B_5\}$.   
\end{itemize}
 \end{example}
 We can express the Specht varieties in terms of the partial orders and orbit types.

\begin{theorem}[\cite{debus_moustrou_riener_verdure,moustrou2021symmetric,woo2005ideals}]\label{thm:SnBnEquivalence}
For Specht ideals of $S_n$ and $B_n$ we have chains of equivalences.
\begin{table}[H]
    \centering
    \begin{tabular}{ccccc}
         $\mu \unlhd_S \lambda $ & $\Leftrightarrow $ & $I_{\mu}^\mathbf{S} \subseteq I_{\lambda}^\mathbf{S}$ &$\Leftrightarrow$ & $V_\lambda^{\mathbf{S}} \subseteq V_\mu^{\mathbf{S}}.$\\
       $(\vartheta,\omega) \unlhd_B (\lambda,\mu)$ & $\Leftrightarrow$ & $I_{(\vartheta,\omega)}^\mathbf{B} \subseteq I_{(\lambda,\mu)}^\mathbf{B}$ & $\Leftrightarrow$ & $V_{(\lambda,\mu)}^\mathbf{B} \subseteq V_{(\vartheta,\omega)}^\mathbf{B}.$
    \end{tabular}
\end{table}

Moreover,

\begin{table}[H]
    \centering
    \begin{tabular}{ccccc}
       $V_\lambda^\mathbf{S}$ & $=$ & $\bigcup_{\mu \not \unlhd \lambda} O(\mu)$ &$=$ & $\left( \bigcup_{\mu  \unlhd \lambda} O(\mu) \right)^c.$\\
       $ V_{(\lambda,\mu)}^\mathbf{B}$ & $=$ & $\bigcup_{(\vartheta,\omega) \not \unlhd_B (\lambda,\mu)} O(\vartheta,\omega)$ & $=$ & $\left( \bigcup_{(\vartheta,\omega) \unlhd_B (\lambda,\mu)} O(\vartheta,\omega)  \right)^c.$
    \end{tabular}
\end{table}
\end{theorem}
The Specht ideals of the symmetric group are radical \cite{Murai2022Specht, woo2005ideals}, but it is an open question if the ideals for the hyperoctahedral group are radical \cite[Conjecture~8.1]{debus_moustrou_riener_verdure}.
Moreover, the $\mathbf{S}$-Specht ideal $I_\lambda^{\mathbf{S}}$ is Cohen-Macaulay if and only if $\lambda$ is a hook, has length $2$, or $\lambda = (\lambda_1,\lambda_1,1)$ \cite{yanagawa2021specht}.

\section{Partial order}\label{sec:Partial order}

In this section, we introduce a partial order $\unlhd_D$ on the set of dipartitions $\mathcal{D}_n$ for which an analogue of the poset equivalences in \cref{thm:SnBnEquivalence} holds. Note that sets of the form $\{\lambda,\pm\}$ only exist if $n$ is even. 
For instance, we have 
\begin{align*}&\mathcal{D}_2=\{\{(2),\emptyset\},\{(1),+\},\{(1),-\},\{(1,1),\emptyset\}\}, \text{ and } \\ &\mathcal{D}_3 = \{\{(3),\emptyset\},\{(2,1),\emptyset\},\{(1,1,1),\emptyset\},\{(2),(1)\},\{(1,1),(1)\}\}.\end{align*} 
For this purpose, we define an auxiliary poset $\widetilde{\mathcal{D}}_n$ and extend its partial order to $\mathcal{D}_n$.
We denote by $\widetilde{\mathcal{D}}_n$ the set of all multisets $\{\lambda,\mu\}$ with $\lambda \vdash k$ and $\mu \vdash n-k$ for some $0 \leq k \leq n$. If $n$ is odd we have $\mathcal{D}_n = \widetilde{\mathcal{D}}_n$ and otherwise we obtain $\mathcal{D}_n$ from $\widetilde{\mathcal{D}}_n$ by replacing each multiset of the form $\{\lambda,\lambda\}$ in $\widetilde{\mathcal{D}}_n$ with the two distinct sets $\{\lambda,+\}$ and $\{\lambda,-\}$.

\begin{definition}\label{def:DnPoset}
\begin{enumerate}
    \item[(1)] For two multisets $\{\lambda,\mu\},\{\theta,\omega\} \in \widetilde{\mathcal{D}}_n$ we define the relation $\{\lambda,\mu\} \preccurlyeq_D\{\theta, \omega \}$ if 
\[\left((\lambda,\mu) \unlhd_B (\theta,\omega) \text{ or } (\lambda,\mu) \unlhd_B (\omega,\theta) \right) \text{ and } \left((\mu,\lambda) \unlhd_B (\theta,\omega) \text{ or } (\mu,\lambda) \unlhd_B (\omega,\theta) \right). \]
    \item[(2)] We define a binary relation $\unlhd_D$ on $\mathcal{D}_n$ as follows. The sets $\{\lambda,+\}$ and $\{\lambda,-\} \in \mathcal{D}_n$ are incomparable with respect to $\unlhd_D$. For $\lambda \neq \mu$ we define $\{\lambda,\pm\} \unlhd_D \{\mu,\pm\}$ if $\{\lambda,\lambda\} \preccurlyeq_D\{\mu,\mu\}$. Moreover, $\{\lambda,\mu\} \unlhd_D \{\theta,\omega\}$ if $\{\lambda,\mu\} \preccurlyeq_D\{\theta,\omega\}$ and $\{\lambda,\mu\} \left\{ \begin{array}{c}
\unlhd_D \\ 
\unrhd_D
\end{array} \right\} \{\lambda,\pm\}$ if $\{\lambda,\mu\} \left\{ \begin{array}{c}
\preccurlyeq_D \\ 
\succcurlyeq_D
\end{array} \right\}\{\lambda,\lambda\}$.
\end{enumerate}
\end{definition}

\begin{figure}[!htb]
    \centering
    \begin{minipage}{.5\textwidth}
        \centering
    \begin{tikzpicture}
        \node (1) at (0,0) {\footnotesize{$\setn{\emptyset ,(1,1,1,1)}$}}; 
        \node (2) at (0,1.57) {\footnotesize{$\setn{(1),(1,1,1)}$}};
        \node (3) at (-2,3.14) {\footnotesize{$\setn{\emptyset ,(2,1,1)}$}};
        \node (4) at (0,3.14) {\footnotesize{$\setn{(1,1),+}$}};
        \node (5) at (2,3.14) {\footnotesize{$\setn{(1,1),-}$}};
        \node (6) at (-1,4.71) {\footnotesize{$\setn{\emptyset ,(2,2)}$}};
        \node (7) at (1,4.71) {\footnotesize{$\setn{(2),(1,1)}$}};
        \node (8) at (0,6.29) {\footnotesize{$\setn{(1),(2,1)}$}};
        \node (9) at (-2,7.86) {\footnotesize{$\setn{\emptyset ,(3,1)}$}};
        \node (10) at (0,7.86) {\footnotesize{$\setn{(2),+}$}};
        \node (11) at (2,7.86) {\footnotesize{$\setn{(2),-}$}};
        \node (12) at (0,9.43) {\footnotesize{$\setn{(1),(3)}$}};
        \node (13) at (0,11) {\footnotesize{$\setn{\emptyset ,(4)}$}};

        \draw[edge] (1) -- (2);
        \draw[edge] (2) -- (3);
        \draw[edge] (2) -- (4);
        \draw[edge] (2) -- (5);
        \draw[edge] (3) -- (6);
        \draw[edge] (3) -- (7);
        \draw[edge] (4) -- (6);
        \draw[edge] (4) -- (7);
        \draw[edge] (5) -- (6);
        \draw[edge] (5) -- (7);
        \draw[edge] (6) -- (8);
        \draw[edge] (7) -- (8);
        \draw[edge] (8) -- (9);
        \draw[edge] (8) -- (10);
        \draw[edge] (8) -- (11);
        \draw[edge] (9) -- (12);
        \draw[edge] (10) -- (12);
        \draw[edge] (11) -- (12);
        \draw[edge] (12) -- (13);
    \end{tikzpicture}
    \caption{Hasse diagram of $(\mathcal{D}_4,\unlhd_D)$}

    \label{fig:D_4Poset}
    \end{minipage}%
    \begin{minipage}{0.5\textwidth}
        \centering
            \begin{tikzpicture}
        \node (1) at (0,0) {\footnotesize{$\setn{\emptyset ,(1,1,1,1,1)}$}}; 
        \node (2) at (0,1) {\footnotesize{$\setn{(1),(1,1,1,1)}$}};
        \node (3) at (-1.5,2) {\footnotesize{$\setn{(1,1),(1,1,1)}$}};
        \node (4) at (1.5,2) {\footnotesize{$\setn{\emptyset ,(2,1,1,1)}$}};
        \node (5) at (-1.5,3) {\footnotesize{$\setn{\emptyset ,(2,2,1)}$}};
        \node (6) at (1.5,3) {\footnotesize{$\setn{(2),(1,1,1)}$}};
        \node (7) at (0,4) {\footnotesize{$\setn{(1),(2,1,1)}$}};
        \node (8) at (-1.5,5) {\footnotesize{$\setn{(1,1),(2,1)}$}};
        \node (9) at (1.5,5) {\footnotesize{$\setn{\emptyset ,(3,1,1)}$}};
        \node (10) at (-2,6) {\footnotesize{$\setn{(1),(2,2)}$}};
        \node (11) at (2,7) {\footnotesize{$\setn{(1,1),(3)}$}};
        \node (12) at (-2,7) {\footnotesize{$\setn{\emptyset ,(3,2)}$}};
        \node (13) at (0,7) {\footnotesize{$\setn{(2),(2,1)}$}};
        \node (14) at (0,8) {\footnotesize{$\setn{(1),(3,1)}$}};
        \node (15) at (-1,9) {\footnotesize{$\setn{\emptyset ,(4,1)}$}};
        \node (16) at (1,9) {\footnotesize{$\setn{(2),(3)}$}};
        \node (17) at (0,10) {\footnotesize{$\setn{(1),(4)}$}};
        \node (18) at (0,11) {\footnotesize{$\setn{\emptyset ,(5)}$}};

        \draw[edge] (1) -- (2);
        \draw[edge] (2) -- (3);
        \draw[edge] (2) -- (4);
        \draw[edge] (3) -- (5);
        \draw[edge] (3) -- (6);
        \draw[edge] (4) -- (5);
        \draw[edge] (4) -- (6);
        \draw[edge] (5) -- (7);
        \draw[edge] (6) -- (7);
        \draw[edge] (7) -- (8);
        \draw[edge] (7) -- (9);
        \draw[edge] (8) -- (10);
        \draw[edge] (8) -- (11);
        \draw[edge] (9) -- (11);
        \draw[edge] (9) -- (12);
        \draw[edge] (10) -- (12);
        \draw[edge] (10) -- (13);
        \draw[edge] (11) -- (14);
        \draw[edge] (12) -- (14);
        \draw[edge] (13) -- (14);
        \draw[edge] (14) -- (15);
        \draw[edge] (14) -- (16);
        \draw[edge] (15) -- (17);
        \draw[edge] (16) -- (17);
        \draw[edge] (17) -- (18);
    \end{tikzpicture}
    
        \caption{Hasse diagram of $(\mathcal{D}_5,\unlhd_D)$}
        \label{fig:D5_Poset}
    \end{minipage}
\end{figure}

\begin{lemma} \label{lem: D_n is poset}
The set $(\mathcal{D}_n,\unlhd_D)$ is a poset and $\{\emptyset,(1,\ldots,1)\}$ is the minimal element in $\mathcal{D}_n$.
\end{lemma}
\begin{proof}
First, we verify that the relation $\preccurlyeq_D$ is a partial order on the set $\widetilde{{\mathcal{D}}}_n$.
Recall that $\unlhd_B$ is a partial order on $\mathcal{B}_n$ \cite[Section~4]{debus_moustrou_riener_verdure}. 
Let $\{\lambda,\mu\},\{\theta,\omega\} \in \widetilde{\mathcal{D}}_n$.
\begin{enumerate}
    \item[(i)] Reflexivity and transitivity of $\preccurlyeq_D$ follow immediately from the analogous properties of $\unlhd_B$.
    \item[(ii)] It is left to show that $\preccurlyeq_D$ is anti-symmetric. Suppose we have $\{\lambda,\mu\} \preccurlyeq_D \{ \theta,\omega\} \preccurlyeq_D \{ \lambda,\mu\}$. 
    Without loss of generality we suppose $(\lambda,\mu) \unlhd_B (\theta,\omega)$.
    Then we have 
    \begin{align*}
    (\lambda,\mu) \unlhd_B (\theta,\omega) \unlhd_B (\lambda,\mu) \text{ or }
    (\lambda,\mu) \unlhd_B (\theta,\omega) \unlhd_B (\mu,\lambda).
    \end{align*}
However, the relation $\unlhd_B$ on the set of bipartitions of $n$ is anti-symmetric, since it is a partial order. In the first case we have $(\lambda,\mu)=(\theta,\omega)$ which implies $\{ \lambda,\mu\} = \{ \theta,\omega\}$. 
If the second chain of relations holds we have
    \begin{align*}
    (\lambda,\mu) \unlhd_B (\theta,\omega) \unlhd_B (\mu,\lambda) \unlhd_B (\theta,\omega) \text{ or }
    (\lambda,\mu) \unlhd_B (\theta,\omega) \unlhd_B (\mu,\lambda) \unlhd_B (\omega,\theta). 
    \end{align*}
    Again, the first case results in $\{\lambda,\mu\}=\{\theta,\omega\}$. In the second case, we have 
        \begin{align*}
    (\lambda,\mu) \unlhd_B (\theta,\omega) \unlhd (\mu,\lambda) \unlhd_B (\omega,\theta)  \unlhd_B (\lambda,\mu) \text{ or }
    (\lambda,\mu) \unlhd_B (\theta,\omega) \unlhd_B (\mu,\lambda) \unlhd_B (\omega,\theta)  \unlhd_B (\mu,\lambda). 
    \end{align*}
    and both cases imply $\{\lambda,\mu\}=\{ \theta,\omega\}$. 
\end{enumerate} 
Finally, note that going from the poset $(\widetilde{\mathcal{D}}_n,\preccurlyeq_D)$ to the poset $(\mathcal{D}_n,\unlhd_D)$ means that we replace nodes of the form $\{\lambda,\lambda\}$ by two mutually incomparable nodes, and they have the same parent and children nodes in the Hasse diagram.
Reflexivity and antisymmetry remain clearly true and any directed cycle must use both nodes $\{\lambda,+\}$ and $\{\lambda,-\}$. 
This cannot be the case since they are defined as incomparable in $(\mathcal{D}_n,\unlhd_D)$. \\
Observe that for every bipartition $(\lambda,\mu)$ of $n$ with $l:=\len(\lambda)\geq \len(\mu)=:k$ we have
$\lambda_{j+1} + \sum_{i=1}^j (\lambda_i+\mu_i) \geq j+1$ for all $1 \leq j < l$ and for $j \geq l$ the sum is $n$.
Thus, we have $(\emptyset,(1,\ldots,1)) \unlhd_B ((1,\ldots,1),\emptyset) \unlhd_B (\lambda,\mu)$ which shows that $\{\emptyset,(1,\ldots,1)\}$ is indeed the minimal element in $\mathcal{D}_n$.
\end{proof}

The following technical proposition will be used in our proof of the equivalence of posets for type $D$ in \cref{thm:main}. We implicitly assume that $n$ is even.

\begin{proposition} \label{prop:Covering lambda lambda}
 Let $\{\lambda, \pm\} \unrhc_D \Theta$ be a covering relation in the poset $(\mathcal{D}_n,\unlhd_D)$. 
 Then there exist an integer $1 \leq p \leq \len (\lambda)=:l$, partitions $\theta = (\lambda_1,\ldots,\lambda_{p-1},\lambda_p-1,\lambda_{p+1},\ldots,\lambda_l)$ and $\omega =  (\lambda_1,\ldots,\lambda_{p},\lambda_{p+1}+1,\lambda_{p+2},\ldots)$ such that $\Theta = \{\theta,\omega\}$.
\end{proposition}
Note that we must have $\lambda_p > \lambda_{p+1}$. 
Only in this case $\theta$ and $\omega$ are indeed partitions.
Observe that $(\lambda,\lambda)\unrhc_B (\omega,\theta) \unrhc_B (\theta,\omega)$ holds.
In particular, we can conclude that if $\{\lambda,\pm\}$ covers $\Theta \in \mathcal{D}_n$, then $\Theta = \{\theta,\omega\}$ for distinct partitions $\theta$ and $\omega$.
We point out that, when identifying dipartitions with multisets of two diagrams, any such covering described in \cref{prop:Covering lambda lambda} corresponds to moving the last box from row $p$ of one of the diagrams of shape $\lambda$ to the other diagram into row $p+1$. 
We visualize this in the following example.

\begin{example}
In the poset $(\mathcal{D}_6,\unlhd_D)$ the dipartition $\{(2,1),\pm\}$ covers two dipartitions. 
Using diagram notation we have $\left\{\ytableausetup{smalltableaux,centertableaux}
\begin{ytableau}
\empty & *(red)  \\
*(violet) & \none
\end{ytableau},\pm\right\} $ covers the two dipartitions $ \left\{\ytableausetup{smalltableaux,centertableaux}
\begin{ytableau}
\empty   \\
*(white)
\end{ytableau}~ , ~\begin{ytableau}
\empty & \empty   \\
\empty & *(red)
\end{ytableau} \right\}$, where $(p,q) =(1,2)$,
and 
$\left\{
\ytableausetup{smalltableaux,centertableaux}
\begin{array}{cc}
\begin{ytableau}
\empty & \empty  \\
\none \\
\none
\end{ytableau}
~,~&
\begin{ytableau}
\empty & \empty  \\
\empty & \none \\
*(violet) & \none
\end{ytableau}
\end{array}
\right\}$, where $(p,q)=(2,3)$.
\end{example}

\begin{proof}[Proof of Proposition \ref{prop:Covering lambda lambda}]
We show that whenever we have $\{\lambda,\pm\} \unrhc_D \{\theta',\omega'\}$ in the poset $(\mathcal{D}_n,\unlhd_D)$, there exists a dipartiton $\{\theta,\omega\}$ as described in the \cref{prop:Covering lambda lambda} which satisfies $\{\lambda,\pm\} \unrhc_D  \{\theta,\omega\} \unrhd_D \{\theta',\omega'\}$. 
Thus, any cover must be of this form.
It sufficies to assume $\theta' \neq \omega'$, because a dipartition $\{\rho,\pm\}$ cannot be the minimal element in the poset by \cref{lem: D_n is poset}. \\
Suppose we have $\{\lambda,\pm\} \unrhc_D \{\theta',\omega'\}$. 
Let $l\coloneqq \len (\lambda)$, $p'\coloneqq  \min\{ i \, \, : \, \, \theta_i' \neq \lambda_i ~ \text{or} ~ \omega_i' \neq \lambda_i\}$ and without loss of generality we can suppose $\lambda_{p'} > \theta_{p'}'$. 
Let $p \geq p'$ be minimal with $\lambda_p > \lambda_{p+1}$.
We define $\theta\coloneqq  (\lambda_1,\ldots,\lambda_{p-1},\lambda_p-1,\lambda_{p+1},\ldots,\lambda_l)$ and $\omega \coloneqq  (\lambda_1,\ldots,\lambda_{p},\lambda_{p+1}+1,\lambda_{p+2},\ldots)$. 
Since $\theta, \omega$ are partitions of $\frac{n}{2}-1$ and $\frac{n}{2}+1$, respectively, we have $\{\theta,\omega\} \in \mathcal{D}_n$ and we verify $\{\theta',\omega'\} \preccurlyeq_D \{\theta,\omega\} \preccurlyeq_D \{\lambda,\lambda\}$. 
But the latter inequality can be seen from $(\theta,\omega) \unlhc_B (\omega,\theta) \unlhc_B (\lambda,\lambda)$. 
So we verify $\{\theta',\omega'\} \preccurlyeq_D \{\theta,\omega\}$. 
We claim that
\begin{align*}
    \sum_{i=1}^k (\theta_i'+\omega_i') & \leq \sum_{i=1}^k (\theta_i+\omega_i), \\
    \max\{\omega_{k+1}',\theta_{k+1}'\} + \sum_{i=1}^k (\omega_i'+\theta_i') & \leq \omega_{k+1} + \sum_{i=1}^k (\theta_i+\omega_i)  
\end{align*}
hold for every integer $k$. 
If $k<p$ or $k\geq p+1$ the right hand side sums equal those of $(\lambda,\lambda)$ and therefore the inequalities follow from $\{\lambda,\pm\} \unrhd_D \{\theta',\omega'\}$. \\
For $k=p$ we have $\theta_k' \leq \lambda_p-1 = \theta_k$ and $\omega_k' \leq \lambda_p = \omega_k$ which verifies that both inequalities also hold. 
This concludes the proof, as $\{\theta,\omega\}$ is of the claimed form. 
\end{proof}
The following corollary will be used in \cref{sec:PosetD_n}.

\begin{corollary}\label{cor:dominance and didomiance}
Let $\lambda \uplus \lambda \unrhc_S \kappa $ be a cover in the poset $\mathcal{P}_n$. Then $\kappa$ is the fusion of partitions $\theta \vdash \frac{n}{2}-1, \omega \vdash \frac{n}{2}+1$ for which $\{\lambda,\pm\} \unrhc_D \{ \theta, \omega\}$ is a cover in the poset $\mathcal{D}_n$. 
\end{corollary}
\begin{proof}
Since $\lambda \uplus \lambda$ covers $\kappa$ there are integers $i < j$ such that $2\lambda_i-1 = \kappa_i, 2\lambda_j+1 = \kappa_j$, and $2\lambda_k = 2\lambda_i-1$ for all $i < k < j$, and $2\lambda_k = \kappa_k$ holds for all $k \not\in \{i,j\}$ by \cite[Proposition~2.3]{brylawski1973lattice}.
Observe that $2\lambda_k$ is even while $2\lambda_i-1$ is odd which implies $j = i+1$.
Consider the dipartition $\{\theta,\omega\}$ with $\theta =(\lambda_1,\ldots,\lambda_{i-1},\lambda_i-1,\lambda_{i+1},\ldots)$ and $\omega = (\lambda_1,\ldots,\lambda_{i},\lambda_{i+1}+1, \lambda_{i+2},\ldots)$.
Observe that $\kappa = \theta \uplus \omega$ and we have $\{\lambda,\pm\} \unrhc_D \{\theta,\omega\}$.
We claim that the latter relation is a cover in $\mathcal{D}_n$.
This follows now from \cref{prop:Covering lambda lambda}, since if $\{\lambda,\pm \} \unrhc_D \{\theta', \omega'\} \unrhd_D \{\theta,\omega\}$ holds and the first relation is a cover, then $\theta' \neq \omega'$ are partitions for which $\lambda \uplus \lambda \unrhc_S \theta' \uplus \omega' \unrhd_S \omega \uplus \theta = \kappa$. 
But this implies that $\kappa = \theta'\uplus \omega'$. 
Again applying \cref{prop:Covering lambda lambda} shows that indeed $\{\theta', \omega'\} = \{\omega, \theta\}$.
\end{proof}

\section{The posets of $\mathbf{D}$-Specht ideals and their varieties}\label{sec:PosetD_n}
Our main results for $\mathbf{D}$-Specht ideals are that the same chain of equivalences holds as for $\mathbf{S}$- and $\mathbf{B}$-Specht ideals (\cref{thm:main}), and that the $\mathbf{D}$-Specht varieties cannot be expressed as a disjoint union of orbit sets using the partial order $\unlhd_D$ and, in fact, no partial order would work (\cref{thm:nounionanalogue}).
Nevertheless, a weaker characterization of the varieties in terms of $\mathbf{B}$-orbit sets and additional points in $\KK^*$ with certain formal sign structure defined in \cref{def:sign} (see \cref{prop:Specht varieties I}).

\begin{theorem} \label{thm:main}
Let $\Lambda,\Theta \in \mathcal{D}_n$. 
The following assertions are equivalent in characteristic zero:
 \begin{enumerate}
    \item[(A)] $\Theta \unlhd_D \Lambda$.
    \item[(B)] $ I_\Theta^{\mathbf{D}}\subseteq I_\Lambda^{\mathbf{D}}$.
    \item[(C)] $V_\Lambda^\mathbf{D} \subseteq V_\Theta^\mathbf{D}$. 
\end{enumerate}   
\end{theorem}

We divide the proof of \cref{thm:main} into two parts. 
Our strategy, as already used in \cite{debus_moustrou_riener_verdure, moustrou2021symmetric} is to show that (A) implies (B) (\cref{lem:cover implies ideal inclusion}) and that (C) implies (A) (\cref{lem:(C) implies (A)}). Since naturally (B) implies (C) this is already sufficient.
Observe that radicality of the ideals $I_\Theta^{\mathbf{D}}$ would immediately imply the equivalence of (B) and (C). 

\begin{remark}
Although, the implication (B) $\Rightarrow$ (C) and, as we will see in \cref{rem:finite characteristic}, the implication (A) $\Rightarrow$ (B) do also hold if $\operatorname{char}(\KK) > 0$, the implication (C) $\Rightarrow$ (A) might fail.
For instance, for $n=2$ and $\KK = \ZZ / 2 \ZZ$ the ideals $I_{\{(1),\pm\}}^\mathbf{D} = (X_1+X_2)$ and $I_{\{(1,1),\emptyset\}}^\mathbf{D} = ((X_1+X_2)^2)$ are distinct but their varieties are equal. 
\end{remark}

\begin{lemma} \label{lem:cover implies ideal inclusion}
Let $\Lambda,\Theta \in \mathcal{D}_n$ with $\Theta \unlhd_D \Lambda$. 
Then $ I_\Theta^{{\mathbf{D}}}\subseteq I_\Lambda^{\mathbf{D}}$.
\end{lemma}
\begin{proof}
\begin{itemize}
    \item We first suppose that $\Lambda = \{\lambda,\mu\}$ for distinct partitions $\lambda,\mu$. 
If $\Theta = \{\theta,\omega\}$ for partitions $\theta\neq\omega$ we have, by the definition of $\unlhd_D$ that $(\theta,\omega) \unlhd_B (\lambda,\mu)$ or $(\theta,\omega) \unlhd_B (\mu,\lambda)$, and the same holds for $(\omega,\theta)$.
From the equivalence of the posets for $\mathbf{B}$-Specht ideals (\cref{thm:SnBnEquivalence}) we deduce
$ I_\Theta^{\mathbf{D}} = I_{(\theta,\omega)}^{\mathbf{B}}+I_{(\omega,\theta)}^{\mathbf{B}} \subseteq I_{(\lambda,\mu)}^{\mathbf{B}}+I_{(\mu,\lambda)}^{\mathbf{B}} = I_\Lambda^{\mathbf{D}} $
which was to show. 
On the other hand, if $\Theta = \{\theta,\pm\}$ we have $I^{\mathbf{D}}_\Theta \subseteq I_{(\theta,\theta)}^{\mathbf{B}} \subseteq I_{(\lambda,\mu)}^{\mathbf{B}}+I_{(\mu,\lambda)}^{\mathbf{B}} = I_\Lambda^{\mathbf{D}}$. 

\item Second, we suppose that $\Lambda = \{\lambda,\pm\}$. Without loss of generality, let $\Lambda \unrhc_D \Theta$ a cover, since otherwise there is a dipartition $\setn{\mu,\nu}$ with $\mu \neq \nu$ and $\Lambda \unrhc_D \{\mu,\nu\} \unrhd_D \Theta$ by \cref{prop:Covering lambda lambda}.
As argued in the previous paragraph we then have $I_{\{\mu,\nu\}}^{\mathbf{D}} \supseteq I_\Theta^{\mathbf{D}}$. \\
Suppose that $\Lambda = \{\lambda,\pm\}$ covers $\Theta = \{\theta,\omega\}$. 
Then by \cref{prop:Covering lambda lambda}, without loss of generality, $(\omega,\theta)$ can be obtained from $(\lambda,\lambda)$ by moving one box from the right hand side diagram to the left diagram in a row below (see \cref{fig:visualtion in proof of moving boxes} for a visualization). 
In particular, $\Theta $ is not of the form $\{\theta,\pm\}$.
It is sufficient to prove that for both $(\omega,\theta)$ and $(\theta,\omega)$ there is a $\mathbf{B}$-Specht polynomial of the respective shape which is contained in $I_{\{\lambda,\pm\}}^{\mathbf{D}}$. This is because every other $\mathbf{B}$-Specht polynomial of this shape is obtained by permuting the variables and is thus also contained in $I_{\{\lambda,\pm\}}^{\mathbf{D}}$, as the ideal is closed under the action of $S_n \subseteq D_n$. 
Furthermore, since $(\omega,\theta) \unrhd_B (\theta,\omega)$ we have $I_{(\omega,\theta)}^{\mathbf{B}} \supseteq I_{(\theta,\omega)}^{\mathbf{B}}$ by \cref{thm:SnBnEquivalence}. 
So we only need to show that one $\mathbf{B}$-Specht polynomial of shape $(\omega,\theta)$ lies in the ideal $I_{\{\lambda,\pm\}}^{\mathbf{D}}$. \\
We follow the construction of $(\omega,\theta)$ from $(\lambda,\lambda)$ by moving one box.
Let $(T,S)$ be a bitableau of shape $(\lambda,\lambda)$ and let $k$ be the integer in the box that is at another position in the bidiagram $(\omega,\theta)$. Let $(T',S')$ be the bitableau of shape $(\omega,\theta)$ where any integer $k \neq i \in [n]$ is in the same box as in $(T,S)$.
We focus on the changed columns.  
Let $B_1$ be the sequence of integers in the column of $T$ that we enlarge, $A$ be the singleton of the integer of the box that we move and $A \cup B_2$ be the integers in the column of $S$ which originally contains the box of the integer in $A$.
We use set notation for $A,B_1$ and $B_2$ and suppose, without loss of generality, that their elements are ordered increasingly. 
This is because the order of the elements only affects signs of the corresponding Specht polynomials and consequently does not affect ideal inclusion.\\ 
Without loss of generality we can suppose $A=\{1\}, B_1 = \{2,3,\ldots,b+1\},B_2 = \{b+2,\ldots,2b\}$. We have $|A|=1, |B_1|=b, |B_2|=b-1$ and we consider the following constituents of Specht polynomials 
\begin{align*}
P(\underline{X}) = &\Delta_{A \cup B_1}(\underline{X}^2)\Delta_{B_2}(\underline{X}^2)\prod_{i \in B_2}X_i, \\
Q_1(\underline{X}) =& \Delta_{B_1}(\underline{X}^2)\Delta_{A \cup B_2}(\underline{X}^2) \prod_{i \in B_1} X_i, \\
Q_2(\underline{X}) = &\Delta_{B_1}(\underline{X}^2)\Delta_{A \cup B_2}(\underline{X}^2) \prod_{i \in A \cup B_2} X_i, \\
Q_\pm(\underline{X}) = & Q_1(\underline{X}) \pm Q_2(\underline{X}),
\end{align*}
and we will show that the ideal generated by the $S_{2b}$-orbit of $Q_\pm$ contains $P$ in the smaller polynomial ring $\KK[X_1,\ldots,X_{2b}]$.
Note that $\deg(P)=\deg(Q)+1$. 
We prove the identity 
\begin{align} \label{eq:C}
  \sum_{\sigma \in S_{A \cup B_1}}\operatorname{sgn}(\sigma)\sigma Q_{2}(\underline{X}) X_1 = C \cdot P(\underline{X})   
\end{align}
in the \hyperref[appendix]{Appendix} in \cref{prop:Ccalculataion} and also calculate $C=b!$.
Here we only use $C \neq 0$.
We now show that
\begin{align} \label{eq:3}
\sum_{\sigma \in S_{A \cup B_1}}\operatorname{sgn}(\sigma)\sigma Q_{1}(\underline{X})  X_1 = 0    
\end{align}
which implies that $P$ lies in the linear span of the $S_{b+1}$-orbit of $Q_\pm(\underline{X})  X_1$, because we have $\sum_{\sigma \in S_{A \cup B_1}}\operatorname{sgn}(\sigma)\sigma Q_{\pm}(\underline{X})  X_1 = C \cdot P(\underline{X})  $. 
Observe that $\spe_{(T',S')}^{\mathbf{B}} = P\cdot F$ for a polynomial $F \in \KK[X_{2b+1},\ldots,X_n]$, and 
$\spe_{(T,S)}^{\mathbf{D}} = \spe_{(T,S)}^{\mathbf{B}}\pm \spe_{(S,T)}^{\mathbf{B}} = Q_2 \cdot F \pm Q_1 \cdot G  $ for some $G\in \KK[X_{2b+1},\ldots,X_n]$. Then we have \begin{align*}
    \sum_{\sigma \in S_{A \cup B_1}}\operatorname{sgn}(\sigma)\sigma (\spe^{\mathbf{D}}_{(T,S)}(\underline{X})  X_1) = C\cdot \spe^{\mathbf{B}}_{(T',S')}(\underline{X}) +0 = C \cdot \spe^{\mathbf{B}}_{(T',S')}(\underline{X}) . 
\end{align*}
It remains to prove (\ref{eq:3}).
We have $Q_{1}(\underline{X}) X_1 = \Delta_{B_1}(\underline{X}^2)\Delta_{A\cup B_2}(\underline{X}^2)\prod_{i\in A\cup B_1}X_i$. 
We write $Q_{1}(\underline{X}) X_1$ as $\Delta_{B_1}(\underline{X}^2)\prod_{j \in B_2}(X_1^2-X_j^2)R(\underline{X})\prod_{i \in A \cup B_1}X_i$, where $R(\underline{X})$ is a polynomial in variables $X_j$, for $j \in B_2$. 
So when we act with $S_{A\cup B_1}$ on $Q_{1}(\underline{X}) X_1$, we can disregard the factor $R(\underline{X})\cdot \prod_{i\in A \cup B_1}X_i$ since it is invariant. 
Let us consider $\Delta_{B_1}(\underline{X}^2)\prod_{j \in B_2}(X_1^2-X_j^2)$ but we replace $\underline{X}^2$ by $\underline{X}$ for simplicity.
We have \[ \qquad \quad \prod_{j \in B_2}(X_1-X_j)=X_1^{b-1}-X_1^{b-2}e_1(X_{b+2},\ldots,X_{2b+1})\pm \ldots + (-1)^{b-1}e_{b-1}(X_{b+2},\ldots,X_{2b+1}) \]
where $e_j = \sum_{J \subseteq [b-1]}\prod_{i \in J}X_i$ denotes the $j$-th elementary symmetric polynomial in $b-1$ variables.
Again, the elementary symmetrics in variables from $B_2$ are invariant with respect to $S_{A\cup B_1}$. 
Now, we have that every monomial in  $\Delta_{B_1}(\underline{X})$ is of the form $X_2^{\alpha_2}\cdots X_{b+1}^{\alpha_{b+1}}$ of degree $\frac{(b-1)b}{2}$, and $\alpha_i < b$ for any $2 \leq i \leq b+1$. 
Thus, if we consider a monomial $q$ in 
$X_2^{\alpha_2}\cdots X_{b+1}^{\alpha_{b+1}}(X_1^{b-1}-X_1^{b-2}e_1(X_{b+2},\ldots,X_{2b+1})\pm \ldots + (-1)^{b-1}e_{b-1}(X_{b+2},\ldots,X_{2b+1}))$, then by the pigeonhole principle, there must exist $1 \leq i < j \leq b+1$ with $\alpha_i = \alpha_j$, where $\alpha_1$ is the exponent of $X_1$ in $q$:
The only way that $\alpha_2,\ldots,\alpha_{b+1}$ are pairwise distinct is if $[b-1]\cup\{0\} = \{\alpha_i\, \, : \, \, i\in\{2,\ldots,b+1\}\}$.
But then together with $\alpha_1$, we must have at least two equal exponents $\alpha_i$ and $\alpha_j$. 
We have $$\sum_{\sigma \in S_{A\cup B_1}}\operatorname{sgn}(\sigma)\sigma q = \sum_{\sigma \in S_{A\cup B_1}}\operatorname{sgn}(\sigma)\sigma \cdot (i,j) q = - \sum_{\sigma \in S_{A\cup B_1}}\operatorname{sgn}(\sigma)\sigma q$$
which shows (\ref{eq:3}).
\end{itemize}
\end{proof}

\begin{figure}[h!]
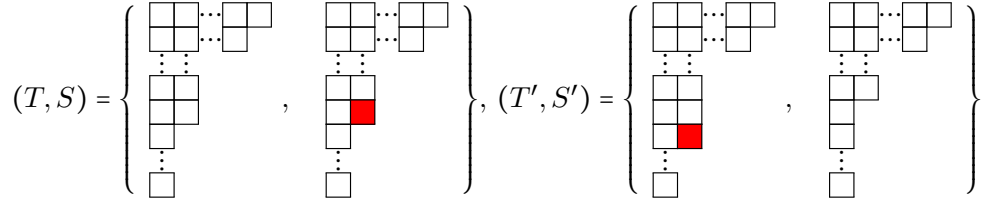

    \centering
    $(T,S) = \left\{
\ytableausetup{boxsize=.8em,centertableaux}
\begin{array}{cc}
\begin{ytableau}
        \empty & \empty & \none[\cdots] & \empty & \empty \\
         \empty & \empty & \none[\cdots]& \empty  & \none \\ 
         \none[\vdots] & \none[\vdots] & \none & \none \\
         \empty & \empty & \none & \none  \\
        \empty & \empty & \none & \none  \\
        \empty & \none \\
        \none[\vdots] & \none \\
        \empty & \none
\end{ytableau}
~,~&
\begin{ytableau}
       \empty & \empty & \none[\cdots] & \empty & \empty \\
         \empty & \empty & \none[\cdots]& \empty  & \none \\ 
         \none[\vdots] & \none[\vdots] & \none & \none \\
         \empty & \empty & \none & \none  \\
        \empty & *(red)  & \none & \none  \\
        \empty & \none \\
        \none[\vdots] & \none \\
        \empty & \none
\end{ytableau}
\end{array} 
\right\}$, $(T',S') = \left\{
\ytableausetup{boxsize=.8em,centertableaux}
\begin{array}{cc}
\begin{ytableau}
        \empty & \empty & \none[\cdots] & \empty & \empty \\
         \empty & \empty & \none[\cdots]& \empty  & \none \\ 
         \none[\vdots] & \none[\vdots] & \none & \none \\
         \empty & \empty & \none & \none  \\
        \empty & \empty & \none & \none  \\
        \empty & *(red) \\
        \none[\vdots] & \none \\
        \empty & \none
\end{ytableau}
~,~&
\begin{ytableau}
       \empty & \empty & \none[\cdots] & \empty & \empty \\
         \empty & \empty & \none[\cdots]& \empty  & \none \\ 
         \none[\vdots] & \none[\vdots] & \none & \none \\
         \empty & \empty & \none & \none  \\
        \empty & \none  & \none & \none  \\
        \empty & \none \\
        \none[\vdots] & \none \\
        \empty & \none
\end{ytableau}
\end{array} 
\right\}$ 
    \caption{$(T,S)$ is of shape $(\lambda,\lambda)$ and $(T',S')$ of shape $(\omega,\theta)$. The red box in the second column of $S$ is moved to the end of the second column of $T$ to obtain $(T',S')$}
    \label{fig:visualtion in proof of moving boxes}
\end{figure}

\begin{remark}\label{rem:finite characteristic}
Our proof of the computation of $C=b!$ in \cref{prop:Ccalculataion} reveals that for every $\sigma \in S_{B_1}$ we have
$\sgn(\sigma)\sigma Q_\pm(\underline{X}) X_1 = Q_\pm(\underline{X}) X_1$ (compare with \cref{eq:C} and \cref{eq:3}).
If we sum over a representative of each left coset in $ S_{A \cup B_1} / S_{B_1}$ instead, we find 
\begin{align}\label{eq:modified sum}
  \sum_{\sigma \in S_{A \cup B_1}/S_{B_1}} \sgn (\sigma) \sigma Q_\pm(\underline{X}) X_1 = P(\underline{X}).  
\end{align}
This does not depend on the explicit choice of the representative of a left coset, since for $\sigma = g_1h_1=g_2h_2$ with $g_1,g_2 \in S_{A \cup B_1}$ and $h_1,h_2 \in S_{B_1}$ we have
\begin{align*} 
 \sgn (g_1)g_1Q_\pm(\underline{X}) X_1  & = \sgn(g_1)g_1 \sgn(h_1)h_1Q_\pm(\underline{X}) X_1 = \sgn (g_1h_1)g_1h_1 Q_\pm(\underline{X}) X_1  \\
 & = \sgn (g_2)g_2 Q_\pm(\underline{X}) X_1. 
\end{align*}
Thus, the modified sum in \cref{eq:modified sum} is well defined. \\
In particular, a slight modification of the proof above also works in finite characteristic. 
This is, since the inclusion in every other case of covering dipartitions works also in finite characteristic (see \cite[Appendix A]{moustrou2021symmetric} for $S_n$ and \cite[Proposition~5.2]{debus_moustrou_riener_verdure} for $B_n$ although this is not explicitly mentioned there).
\end{remark}

We aim to describe $\mathbf{D}$-Specht varieties in terms of $\mathbf{B}$-orbit sets together with certain additional points for which we must introduce a formal sign.
This should generalize the definition of the sign of $x=(x_1,\ldots,x_n) \in (\RR^*)^{2m}$ as the sign of $\prod_{i=1}^{2m} x_i$.
\begin{definition}\label{def:sign}
Let $n$ be even and $x = (x_1,\ldots,x_n) \in (\KK^*)^n$ be such that $x^2$ has $\mathbf{S}$-orbit type $\lambda \uplus \lambda=(2\lambda_1,\ldots,2\lambda_l)$.  
There is $(a_1,\ldots,a_l) \in (\KK^*)^l$ with $a_i^2 \neq a_j^2$ for all $i \neq j$, and up to permutation
\[x = ((-1)^{\sigma_{1,1}}a_1,\ldots,(-1)^{\sigma_{1,2\lambda_1}}a_1,(-1)^{\sigma_{2,1}}a_2,\ldots,(-1)^{\sigma_{2,2\lambda_2}}a_2,\ldots,(-1)^{\sigma_{l,2\lambda_l}}a_l),\]
where $\sigma_{i,j} \in \{0,1\}$.
We define the \emph{formal sign} of $x$ as $\operatorname{sign} (x):= \prod_{i=1}^l \prod_{j=1}^{2\lambda_i} (-1)^{\sigma_{i,j}} \in \{\pm 1\}$
and we define the \emph{$\mathbf{D}$-orbit sets} \[O(\lambda,\pm) := \{x\in (\KK^*)^n \cap O(\lambda \uplus \lambda) \, \, : \, \,  \operatorname{sign}(x)=\pm1\}. \]
\end{definition}
Observe that $a=(a_1,\ldots,a_l)$ is not unique, but the formal sign of $x$ does not depend on the choice of $a$.
This is due to the orbit type $\lambda \uplus \lambda$ of $x$ where each $\pm a_i$ occurs with even multiplicity in $x$.
For instance, $(-1,-1,-1,-1,1,1) \in O(\lambda,+)$ and $(-1,-1,-1,1,1,1) \in O(\lambda,-)$ no matter whether $a_1$ is chosen to be  $1$ or $-1$.

\begin{proposition} \label{prop:Specht varieties I}
The $\mathbf{D}$-Specht varieties corresponding to a dipartition $\Lambda$ can be characterized as follows:
\begin{enumerate}
    \item For $\Lambda = \{\lambda , \mu\}$ with $\lambda \neq \mu$ we have 
\begin{align*}
  V_{\{\lambda,\mu\}}^{\mathbf{D}} \quad  = \quad \bigcup_{\substack{(\omega,\theta) \not \unlhd_B (\lambda,\mu) \text{ and} \\ (\omega,\theta) \not \unlhd_B (\mu,\lambda)}} O{(\omega,\theta)} \quad = \quad \left(\bigcup_{\substack{(\omega,\theta) \unlhd_B (\lambda,\mu) \text{ or } \\ (\omega,\theta)  \unlhd_B (\mu,\lambda) }}O{(\omega,\theta)} \right)^c  . 
  \end{align*}
  \item For $\Lambda = \{\lambda,\pm\}$ we have
\begin{enumerate}
    \item[(i)] $V_{\{\lambda,+\}}^{\mathbf{D}} = V_{(\lambda,\lambda)}^{\mathbf{B}} \cup O(\lambda,-),$
    \item[(ii)] $V_{\{\lambda,-\}}^{\mathbf{D}} = V_{(\lambda,\lambda)}^\mathbf{B} \cup O(\lambda,+).$
\end{enumerate}
\end{enumerate}
\end{proposition}
Observe that the sets $V_{\{\lambda,+\}}^\mathbf{D}$ and $V_{\{\lambda,-\}}^\mathbf{D}$ are not comparable with respect to inclusion which aligns with the incomparability of $\{\lambda,+\}$ and $\{\lambda,-\}$ in the poset $(\mathcal{D}_n,\unlhd_D)$ in \cref{def:DnPoset}.

\begin{proof}
\begin{enumerate}
    \item This follows from the classification of $\mathbf{B}$-Specht varieties based on $\unlhd_B$ and $\mathbf{B}$-orbit sets. 
We have $V_{\{\lambda,\mu\}}^\mathbf{D} = V_{(\lambda,\mu)}^\mathbf{B} \cap V_{(\mu,\lambda)}^\mathbf{B}$ and $V_{(\lambda,\mu)}^\mathbf{B} = \bigcup_{(\omega,\theta) \not \unlhd_B (\lambda,\mu)} O{(\omega,\theta)}$ by \cref{thm:SnBnEquivalence}. 
Therefore, we have 
\begin{align*}
V_{\{\lambda,\mu\}}^\mathbf{D} & = V_{(\lambda,\mu)}^\mathbf{B} \cap V_{(\mu,\lambda)}^\mathbf{B}  \\ 
& = \bigcup_{(\omega,\theta) \not \unlhd_B (\lambda,\mu)}O{(\omega,\theta)}~ \cap \bigcup_{(\omega,\theta) \not \unlhd_B (\mu,\lambda)}O{(\omega,\theta)} =\bigcup_{\substack{(\omega,\theta) \not \unlhd_B (\lambda,\mu) \text{ and} \\ (\omega,\theta) \not \unlhd_B (\mu,\lambda)}} O{(\omega,\theta)}. 
\end{align*}
The second claimed equality follows from the fact that the $\mathbf{B}$-orbit sets are a set partition of $\KK^n$.
\item In our proof we restrict to (i). The proof of (ii) works analogously. 
First, we verify that the right hand side is a subset of the left hand side. 
We have $V_{(\lambda,\lambda)}^\mathbf{B} \subseteq V_{\{\lambda,\pm\}}^\mathbf{D} $ since for $x \in V_{(\lambda,\lambda)}^\mathbf{B}$ and a bitableau $(T,S)$ of shape $(\lambda,\lambda)$ we have 
\[\spe_{(T,S),+}^{\mathbf{D}}(x)= \spe_{(T,S)}^{\mathbf{B}}(x)+\spe_{(S,T)}^{\mathbf{B}}(x)=0+0=0.\]
Next, we consider a point $x \in O(\lambda,-)$.
For every filling $(T,S)$ of the bidiagram of shape $(\lambda,\lambda)$ either a column has two equal entries up to sign, or the fillings of both tableaux are the same (up to signs and permutation within each column).
In the first case we have $\spe_{(T,S)}^{\mathbf{B}}(x)=\spe_{(S,T)}^{\mathbf{B}}(x)=0$ and in the latter we have $\prod_{i \in T}x_i = - \prod_{i \in S}x_i$ since $\operatorname{sign}(x)=-1$. 
Thus $0 = \spe_{T}^{\mathbf{S}}(x^2)\spe_{S}^{\mathbf{S}}(x^2)(\prod_{i \in T}x_i + \prod_{i \in S}x_i) = \spe_{(T,S),+}^{\mathbf{D}}(x)$ which shows $x \in V_{\{\lambda,+\}}^{\mathbf{D}}$.\\
We now show that the left hand side is a subset of the right hand side.
Let $x \in V_{\{\lambda,+\}}^\mathbf{D}$ and suppose that $x_i = 0$ for some integer $i$. We verify that then $x \in V_{(\lambda,\lambda)}^\mathbf{B}$. Let $(T,S)$ be a bidiagram of shape $(\lambda,\lambda)$. Suppose that a box in $S$ contains $i$ with $x_i=0$, then $\spe_{(T,S)}^\mathbf{B}(x)=0$ since $\prod_{i \in S}x_i =0$. Otherwise, we also have $0=\spe_{(T,S)}^\mathbf{D} (x)=\spe_T^{\mathbf{S}}(x^2)\spe_S^{\mathbf{S}}(x^2)(\prod_{i \in T}x_i+\prod_{i \in S}x_i) =  \spe_T^{\mathbf{S}}(x^2)\spe^{\mathbf{S}}_S(x^2) \prod_{i \in S}x_i = \spe_{(T,S)}^\mathbf{B}(x)$, because $\prod_{i \in T}x_i=0$. 
Since $(T,S)$ was arbitrary, this shows $x \in V_{(\lambda,\lambda)}^\mathbf{B}$. \\
Now, let $x \in V_{\{\lambda,+\}}^{\mathbf{D}}\setminus V_{(\lambda,\lambda)}^{\mathbf{B}}$. 
Then we have $x_i \neq 0$ for all $1 \leq i \leq n$ by the previous paragraph. Let $\kappa$ denote the $S_n$-orbit type of $x^2$, i.e. $(\emptyset,\kappa)$ is the $B_n$-orbit type of $x$. 
Since $x \not\in V_{(\lambda,\lambda)}^{\mathbf{B}}$ we have $(\emptyset,\kappa) \unlhd_B (\lambda,\lambda)$ by \cref{thm:SnBnEquivalence}, and thus $\kappa \unlhd_S \lambda \uplus \lambda$.
By \cref{lem:cover implies ideal inclusion} we have $I_\Omega^{\mathbf{D}} \subseteq I_{\{\lambda,+\}}^{\mathbf{D}}$ for all $\Omega \unlhd_D \{\lambda,+\}$. 
In particular, this holds for every $\Omega=\{\omega,\theta\}$ obtained from $\{\lambda,\lambda\}$ by moving one box as in \cref{prop:Covering lambda lambda}. 
Then we have $x\in V_{\{\lambda,+\}}^{\mathbf{D}} \subseteq V_{\{\theta,\omega\}}^{\mathbf{D}} \subseteq V_{(\theta,\omega)}^{\mathbf{B}}$, which shows $(\emptyset,\kappa ) \not \unlhd_B (\theta,\omega)$ and thus $\kappa \not \unlhd_S \theta \uplus \omega$. \\
We will show that $\kappa = \lambda \uplus \lambda$. 
Recall that every partition $\rho$ covered by $\lambda \uplus \lambda$ in the poset $\mathcal{P}_n$ can be achieved as the fusion of the two partitions contained in a dipartition $\Omega \unlhd_D \{\lambda,+\}$ by \cref{cor:dominance and didomiance}. 
Assume that $\kappa \unlhc_S \lambda \uplus \lambda$. 
Then there is a partition $\rho \vdash n$ which is covered by $\lambda \uplus \lambda$ and $\kappa \unlhd_S \rho \unlhc_s \lambda \uplus \lambda$.
However, there is a dipartition $\Omega = \{\theta,\omega\} \unlhd_D \{\lambda,\pm\}$ such that $\rho = \omega \uplus \theta$ which implies $x \not \in V_{\{\theta,\omega\}}^{\mathbf{D}}$ which is a contradiction to $x \in V_{\{ \lambda,+\}}^{\mathbf{D}} \subseteq V_{\{\theta,\omega\}}^{\mathbf{D}}$.
Thus, we must have $\kappa = \lambda \uplus \lambda$.  \\
The claim about the parity on the signs of the coefficients of $x$ follows, since one can fill the bitableau of shape $(\lambda,\lambda)$ with $x^2$ such that every column has pairwise distinct entries. 
Thus the sum of both products of all entries of the tableaux must sum to $0$ in order to have $x \in V_{\{\lambda,+\}}^{\mathbf{D}}$.
\end{enumerate}
\end{proof}
Note that both $\mathbf{B}$-orbit sets $O{(\omega,\theta)}$ and $O{(\theta,\omega)}$ have to be considered in the disjoint union in \cref{prop:Specht varieties I} (1). 
 For instance, for the $\mathbf{D}$-Specht ideal $I_{\{(n-1),(1)\}}^\mathbf{D} = (x_1,\ldots,x_n)$ we have $O{((n),\emptyset)}=\{(0,\ldots,0)\} \subseteq V_{\{(n-1),(1)\}}^\mathbf{D}$ but $O{(\emptyset,(n))}=(\pm a,\ldots,\pm a) \, : \, a \neq 0\} \not \subseteq V_{\{(n-1),(1)\}}^\mathbf{D}$.
 If $O{(\vartheta,\omega)} \neq \emptyset$ and $O {(\omega,\vartheta)} \neq \emptyset$ then we have $\vartheta = \emptyset$ or $\omega = \emptyset$ by \cref{def:orbit types}. \\
Moreover, we can say the following about the set theoretic difference of $V_{(\lambda,\mu)}^\mathbf{B}$ and $V_{(\mu,\lambda)}^\mathbf{B}$.
\begin{observation}
If $x \in V_{(\lambda,\mu)}^\mathbf{B}\setminus V_{(\mu,\lambda)}^\mathbf{B}$ then a coordinate of $x$ must be zero. 
\end{observation}
\begin{proof}
Let $x \in V_{(\lambda,\mu)}^\mathbf{B}\setminus V_{(\mu,\lambda)}^\mathbf{B}$. Then there is a bitableau $(S,T)$ of shape $(\mu,\lambda)$ such that $\spe_{(S,T)}^\mathbf{B}(x) \neq 0$ but $\spe_{(T,S)}^\mathbf{B}(x)=0$. 
The inequality implies that for each pair of distinct integers $i,j$ from a same column in $T$ or $S$ we have $x_i^2 \neq x_j^2$. 
In order for $\spe_{(T,S)}^\mathbf{B}(x)=0$ to hold true, $\prod_{i\in S}x_i$ has to vanish.
In particular we must have $x_i=0$ for some $i\in S$. 
\end{proof}

To prove that (C) implies (A) we use the characterization of $\mathbf{D}$-Specht varieties $V_\{\lambda,\pm\}^{\mathbf{D}}$ in \cref{prop:Specht varieties I} (2).
\begin{lemma} \label{lem:(C) implies (A)}
For $\Lambda,\Theta \in \mathcal{D}_n$ with $V_\Lambda^{\mathbf{D}} \subseteq V_\Theta^{\mathbf{D}}$ the relation $\Theta \unlhd_D \Lambda$ holds.  
\end{lemma}
\begin{proof}
\begin{itemize}
    \item We first suppose that $\Lambda = \{\lambda,\mu\}$ and $\Theta = \{\theta,\omega\}$ for some partitions $\lambda \neq \mu$ and $\theta \neq \omega$ and let $m = \max\{\len (\theta),\len(\omega)\}$. 
We consider the following points in $\KK^n$
\begin{align*}
 P = & (\underbrace{0,\ldots,0}_{\theta_1},\underbrace{1,\ldots,1}_{\omega_1+\theta_2},\ldots,\underbrace{m,\ldots,m\}}_{\omega_{m}+\theta_{m+1}},  \quad 
 Q =  (\underbrace{0,\ldots,0}_{\omega_1},\underbrace{1,\ldots,1}_{\theta_1+\omega_2},\ldots,\underbrace{m,\ldots,m\}}_{\theta_{m}+\omega_{m+1}} \\ R =  &(\underbrace{1,\ldots,1}_{\theta_1+\omega_1},\underbrace{2,\ldots,2}_{\theta_2+\omega_2},\ldots,\underbrace{m,\ldots,m\}}_{\theta_{m}+\omega_{m}},     \end{align*}
 where $\omega_i = 0$ and $\theta_j = 0$ for $i > \len (\omega), j > \len (\theta)$. 
We have $P,Q,R \not\in V_\Theta^{\mathbf{D}}$. 
This can be justified as follows: The bidiagram of shape $(\theta,\omega)$ can be filled with $R$ such that the $i$-th row contains only $i$'s, for all $1 \leq i \leq m$. The bidiagram of shape $(\theta,\omega)$ can be filled with $P$ such that $\theta_i$ contains only $(i-1)$'s and $\omega_i$ only $i$'s, and the bidiagram of $(\omega,\theta)$ can be filled similarly with $Q$. 
Then each column is filled with pairwise distinct entries and there is never a $0$ in the right hand side diagram. 
This implies $P,Q,R \not \in V_\Lambda^{\mathbf{D}}$ by assumption. 
We deduce for all integers $k$:
\begin{itemize}
    \item Since $P \not \in  V_{(\lambda,\mu)}^{\mathbf{B}}$ or $P \not \in  V_{(\mu,\lambda)}^{\mathbf{B}}$ we have $\theta_{k+1} + \sum_{i=1}^k (\theta_i+\omega_i) \leq \lambda_{k+1} + \sum_{i=1}^k (\lambda_i+\mu_i) $ for all $k$ or $\theta_{k+1} + \sum_{i=1}^k (\theta_i+\omega_i) \leq \mu_{k+1} + \sum_{i=1}^k (\lambda_i+\mu_i)$ for all $k$ (because there is a filling of $(\lambda,\mu)$ or $(\mu,\lambda)$ with $P$ such that there is no $0$ in the right hand side tableau and each column contains pairwise distinct entries).
    \item Since $Q \not \in  V_{(\lambda,\mu)}^{\mathbf{B}}$ or $Q \not \in  V_{(\mu,\lambda)}^{\mathbf{B}}$ we have $\omega_{k+1} + \sum_{i=1}^k (\theta_i+\omega_i) \leq \lambda_{k+1} + \sum_{i=1}^k (\lambda_i+\mu_i) $ for all $k$ or $\omega_{k+1} + \sum_{i=1}^k (\theta_i+\omega_i) \leq \mu_{k+1} + \sum_{i=1}^k (\lambda_i+\mu_i) $ for all $k$ (the argumentation goes analogously).
    \item Since $R \not \in V_{(\lambda,\mu)}^{\mathbf{B}} \cap V_{(\mu,\lambda)}^{\mathbf{B}}$ we have $\sum_{i=1}^k (\theta_i+\omega_i) \leq \sum_{i=1}^k (\lambda_i+\mu_i)$ (because there is a filling of $(\theta,\omega)$ with $R$ such that each column contains pairwise distinct entries).
 \end{itemize} 
This shows $\Theta \unlhd_D \Lambda$. 

\item Next, we suppose $\Lambda = \{\lambda,\pm\}$ and $\Theta = \{\theta,\omega\}$ for some partitions $\theta \neq \omega$. 
Since $V_{(\lambda,\lambda)}^{\mathbf{B}} \subseteq V_{\{\lambda,\pm\}}^{\mathbf{D}} \subseteq V_\Theta^{\mathbf{D}}$ we use the points $P,Q,R$ above to show that $\theta_{k+1} + \sum_{i=1}^k (\theta_i+\omega_i) \leq \lambda_{k+1} + 2\sum_{i=1}^k \lambda_i $, $\omega_{k+1} + \sum_{i=1}^k (\theta_i+\omega_i) \leq \lambda_{k+1} + 2\sum_{i=1}^k \lambda_i $ and $\sum_{i=1}^k (\theta_i+\omega_i) \leq 2\sum_{i=1}^k \lambda_i $ hold for all integers $k$. 
We again have $\Theta \unlhd_D \Lambda$. 

\item If $\Lambda =  \{\lambda,\pm\}$ and $\Theta = \{\theta,\pm\}$ we have either $\Lambda = \Theta$ or $\lambda \neq \theta$ according to \cref{prop:Specht varieties I}.
In the latter case we have by \cref{prop:Specht varieties I}
\begin{align*}
 & V_{(\lambda,\lambda)}^{\mathbf{B}} \subseteq V_{\Lambda}^{\mathbf{D}} \subseteq V_{\{\theta,+\}}^{\mathbf{D}} = V_{(\theta,\theta)}^{\mathbf{B}} \cup O(\theta,-), 
 & V_{(\lambda,\lambda)}^{\mathbf{B}} \subseteq V_{\Lambda}^{\mathbf{D}} \subseteq V_{\{\theta,-\}}^{\mathbf{D}} = V_{(\theta,\theta)}^{\mathbf{B}} \cup O(\theta,+).
\end{align*} 
Since the variety $V_{(\lambda,\lambda)}^{\mathbf{B}}$ is closed under sign switching, we have $V_{(\lambda,\lambda)}^{\mathbf{B}} \subseteq V_{(\theta,\theta)}^{\mathbf{B}}$, and thus, by \cref{thm:SnBnEquivalence}, $(\lambda,\lambda) \unrhd_B (\theta,\theta)$.
This proves that indeed $\Theta \unlhd_B \Lambda$ holds.

\item Finally, we consider the case that $\Lambda = \{\lambda,\mu\}$ and $\Theta = \{\theta,\pm\}$. 
We restrict to the case $\Theta = \{\theta,+\}$. The argument goes analogously for $\{\theta,-\}$.
We consider the points $P=Q, R$ as above.
Then we have $P, R \not \in V_{\Theta}^{\mathbf{D}}$ and thus $P, R \not \in V_{\Lambda}^{\mathbf{D}}$ by assumption.
This shows $\sum_{i=1}^k 2\theta_i \leq \sum_{i=1}^k (\lambda_i + \mu_i )$ for all integers $k$, and $\theta_{k+1}+\sum_{i=1}^k 2\theta_i \leq \lambda_{k+1} \sum_{i=1}^k (\lambda_i + \mu_i )$ for all integers $k$ or $\theta_{k+1}+\sum_{i=1}^k 2\theta_i \leq \mu_{k+1} \sum_{i=1}^k (\lambda_i + \mu_i )$ for all integers $k$. 
Thus, we have $\Theta \unlhd_D \Lambda$.
\end{itemize}
\end{proof}

We can now formulate a short proof of \cref{thm:main}.
\begin{proof}[Proof of Theorem \ref{thm:main}]
The implication (A) implies (B) was proven in \cref{lem:cover implies ideal inclusion}.
That (B) implies (C) is clear.
Finally, (C) implies (B) is the statement of \cref{lem:(C) implies (A)}.
\end{proof}

In contrast to the $S_n$ and $B_n$ cases, the $\mathbf{D}$-Specht varieties cannot be described as a disjoint union of orbit sets using the partial order $\unlhd_D$.

\begin{theorem} \label{thm:nounionanalogue}
In general, there does not exist a family of subsets $O({\Lambda})\subseteq \KK^n$, indexed by $\Lambda \in \mathcal{D}_n$, forming a set partition of $\KK^n$, such that for every $\Theta \in \mathcal{D}_n$ the identity $V_\Theta^\mathbf{D} = \bigcup_{\Lambda \not \unlhd_D \Theta} O(\Lambda)$ holds.
\end{theorem}
In fact, the didominance order $\unlhd_D$ is the only partial order $\preccurlyeq$ on $\mathcal{D}_n$ that could satisfy $V_\Theta^\mathbf{D} = \bigcup_{\Lambda \not \preccurlyeq \Theta} O(\Lambda)$ for all $\Theta \in \mathcal{D}_n$.
Note that $V_\Lambda^\mathbf{D} \subseteq V_\Theta^\mathbf{D}$ is equivalent to $\Lambda \unrhd_D \Theta$ by \cref{thm:main} and thus $\Omega \preccurlyeq \Theta$ is equivalent to $\Omega \unlhd_D \Lambda$.
\begin{proof}
We suppose that there is a set partition $(O({\Lambda}))_{\Lambda \in \mathcal{D}_n}$ of $\KK^n$ such that $V_{\Theta}^\mathbf{D} = \bigcup_{\Lambda \not \unlhd_D \Theta} O({\Lambda})$ for every $\Theta \in \mathcal{D}_n$. 
For $n=5$ there are three pairwise incomparable dipartitions $\Lambda= \{(3,2),\emptyset\}$, $\Omega= \{(2),(2,1)\}$ and $\Theta=\{(1,1),(3)\}$ (see \cref{fig:D5_Poset}).
We observe that any dipartition which is strictly smaller than $\Omega$ or $\Theta$ is also smaller than $\Lambda$. 
Moreover, any dipartition strictly larger than $\Lambda$ is also larger than both $\Omega$ and $\Theta$ and for non-zero integers $a, b$ with $a^2 \neq b^2$ we have $(a,a,a,a,b) \in V_\Lambda^\mathbf{D} \setminus (V_\Omega^\mathbf{D} \cup V_\Theta^\mathbf{D})$. 
Therefore, we have $(a,a,a,a,b) \in O(\Omega) \cup O(\Theta)$ since $V_\Lambda^\mathbf{D}$ is the maximal element in the poset of $\mathbf{D}$-Specht varieties with respect to inclusion which contains $(a,a,a,a,b)$ and $\Omega$ and $\Theta$ are the only dipartitions of $5$ incomparable to $\Lambda$.
But since $\Omega$ and $\Theta$ are also incomparable, we have $O(\Theta)\subseteq V_\Omega^\mathbf{D}$ and vice versa. 
This shows that no set partitions depending only on the orbit type and respecting the didominance order can describe the $\mathbf{D}$-Specht varieties. 
\end{proof}

\section{Concluding remarks}\label{sec:conclusion remarks}
We defined the didominance order $\unlhd_D$ in \cref{def:DnPoset} with the purpose of capturing $\mathbf{D}$-Specht ideal inclusions analogously to the $S_n$ and $B_n$ cases. 
An alternative intuitive definition of $\mathbf{D}$-Specht ideals might be to define them as the intersection of the corresponding $\mathbf{B}$-Specht ideals.
However, the ideals $J_{\{\lambda,\mu\}}^\mathbf{D} \coloneqq  I_{(\lambda,\mu)}^\mathbf{B} \cap I_{(\mu,\lambda)}^\mathbf{B}$ and $J_{\{\lambda,\pm\}}^\mathbf{D} \coloneqq  I_{\{\lambda,\pm\}}^\mathbf{D}$ are not necessarily distinct for distinct dipartitions.
For instance, for $n=4$ we have \[J_{\{(2),(1,1)\}}^\mathbf{D} = J_{\{(1),(2,1)\}}^\mathbf{D} = I_{((1),(2,1))}^\mathbf{B} = \left( (X_{\tau(1)}^2-X_{\tau(2)}^2)X_{\tau(1)}X_{\tau(2)}X_{\tau(3)} \, \vert \, \tau \in S_4\right).\]
While it is known that $\mathbf{S}$-Specht ideals are radical \cite{Murai2022Specht, woo2005ideals}, it is an open question whether the $\mathbf{B}$-Specht ideals are radical. 
Computational evidence for $n \leq 6$ and partially for $n=7$ using SAGE \cite{sagemath} suggests that $\mathbf{D}$-Specht ideals are also radical.
However, the radicality of the Specht ideals for $S_n$, and $B_n, D_n$ when $n$ is small might actually be a special case. Most Specht ideals of the dihedral groups are not radical (see \cref{thm:main2dihedral}).
\begin{question}
Are the ideals $I_{\Lambda}^\mathbf{D} \subseteq \KK[\underline{X}]$ radical? 
\end{question}
Furthermore, for any dipartition $\{\lambda,\mu\}$ with $\lambda \neq \mu$ the respective isotypic component in $\KK[\underline{X}]$ is contained in the $\mathbf{D}$-Specht ideal (\cref{cor:isotypic decomposition}), but for $\lambda = \mu$ the question remains open.
\begin{question}
Is the isotypic component of $\{\lambda,\pm\}$ in $\KK[\underline{X}]$ contained in $I_{\{\lambda,\pm\}}^\mathbf{D}$?
\end{question}
For the sequence of complex reflection groups of types $G(r,p,n)$ ($r,p,n \geq 1$ and $p  \mid  r$), which are groups of permutation matrices whose non-zero entries are certain powers of a primitive $r$-th root of unity, (higher) Specht polynomials have been defined in \cite{morita1998higher}.
The sequence contains $S_n$ for $r=p=1$, $B_n$ for $r=2, p=1$ and $D_n$ for $r=p=2$.
A question is therefore if Specht ideals and varieties in this general setting give analogous results to the $S_n,B_n$ and $D_n$ cases.
However, in general the varieties of distinct representations can be equal.
Consider for instance the irreducible representations of $G(r,1,n)$ which correspond to $r$-tuples of integer partitions whose sizes sum to $n$.
Already for $r=3$ and $n=2$ we find that \[
\ytableausetup{smalltableaux,centertableaux} \setlength{\delimitershortfall}{-5pt} \left( \emptyset, \emptyset, \begin{ytableau} {1} & {2}  \end{ytableau}\right) \text{ and } \left( \emptyset, \begin{ytableau} {1} & {2}  \end{ytableau}, \emptyset\right)
\]
have Specht ideal varieties $ V(X_1^2X_2^2) \text{ and } V(X_1X_2)$ which are equal although the Specht ideals are distinct. 
Thus \cref{thm:SnBnEquivalence} and \cref{thm:main} cannot hold in general for the complex reflection groups $G(r,p,n)$.
However, we think that an appropriate order can be defined such that analogous statements to $(A) \Leftrightarrow (B) \Rightarrow (C)$ in \cref{thm:main} still hold.

\section*{Acknowledgment}
We thank Thomas Jahn for fruitful discussions and helpful remarks.

\printbibliography
 \newpage

\begin{appendix} \label{Appendix}

\section{} \label{appendix}
\section*{Proof of \cref{observation1}}
We present a proof of \cref{observation1}. 
\begin{proof}[Proof of \cref{observation1}] \label{proof of Observation1}
Suppose $\spe_{(T,S),+}^{\mathbf{D}} \mapsto f = \sum_{\alpha \in \NN^n} c_\alpha X_1^{\alpha_1}\cdot \ldots \cdot X_n^{\alpha_n}$ defines a $D_n$-equivariant isomorphism. 
For $k \neq l \in [n]$ consider the map $\sigma_{k,l} : \KK[\underline{X}] \to \KK[\underline{X}]$ which only changes the signs of $X_k$ and $X_l$ and the map $\tau_{k,l}  :\KK[\underline{X}] \to \KK[\underline{X}]$ induced by the transposition $(k~l)$. 
Then
$\sigma_{k,l}(f) = f$ if $k,l \in T$ or $k,l \in S$ (Case (1)) and otherwise $\sigma_{k,l}(f)=-f$ (Case (2)), and $\tau_{k,l}(f)=-f$ if $k,l$ are in the same column of $T$ or $S$.
\begin{itemize}
    \item[(a)] First, we exploit the consequences of $\sigma_{k,l}$ which can only change the signs of monomials in $f$.
In Case (2) each monomial $\underline{X}^\alpha\coloneqq X_1^{\alpha_1}\cdot \ldots \cdot X_n^{\alpha_n}$ in $f$ must have at least one odd exponent. 
If $\alpha_i$ is odd for $i \in T$, Case (1) implies that then every $ \alpha_j$ with $j \in T$ must have odd exponent in $\underline{X}^\alpha$.
Furthermore, Case (2) shows then that every $X_j$ with $j \in S$ must have even exponent in $\underline{X}^\alpha$.
In particular, every monomial in $f$ is divisible by $\prod_{i \in T}X_i$ or $\prod_{i \in S}X_i$ and for every monomial $\underline{X}^\alpha$ either all $\alpha_i$ for $i \in T$ are odd and $\alpha_i$ for $i \in S$ are even, or vice versa. Therefore, $f$ is of the form 
\begin{align} \label{eq:form D_n equivariant}
f(\underline{X}) = h_1(\underline{X}^2)\prod_{i \in T}X_i +  h_2(\underline{X}^2)\prod_{j \in S}X_j
\end{align}
with $h_1,h_2 $ being $n$-variate polynomials with coefficients in $\KK$ and $\tau \cdot h_1(\underline{X}^2) = h_2(\underline{X}^2)$.
This is, since $\tau (f) = f$, $\tau(\prod_{i \in T}X_i)=\prod_{i \in S}X_i$ and the parity of the exponents show that by applying $\tau$ each monomial switches from $h_1$ to $h_2$ and vice versa.
\item[(b)]
Moreover, using $\tau_{k,l}$ for $k\neq l$ within the same column of $T$ or $S$ we find $\tau_{k,l}(f)+f=0$ and thus setting $X_k^2 = X_l^2$ gives $h_1(\underline{X}^2)\prod_{i \in T}X_i + h_2(\underline{X}^2)\prod_{i \in S}X_i =0$. 
Note that both $h_1(\underline{X}^2)\prod_{i \in T}X_i $ and $h_2(\underline{X}^2)\prod_{i \in S}X_i $ must be zero since the support of these polynomials has empty intersection.
Thus both $h_1(\underline{X}^2)$ and $h_2(\underline{X}^2)$ are zero whenever $k\neq l$ are within the same column of $T$ or $S$.
Then $X_k^2-X_l^2$ divides these polynomials, which shows that $h_1(\underline{X}^2)$ and $h_2(\underline{X}^2)$ are divisible by $\spe_T^{\mathbf{S}}(\underline{X}^2)\spe_S^{\mathbf{S}}(\underline{X}^2)$.
Combining the observations we find that $f$ is indeed of the form as claimed in \cref{observation1}.
\end{itemize}
An analogous argument shows that the same holds for equivariant isomorphisms $\spe_{(T,S),-}^\mathbf{D} \mapsto f \in \KK[\underline{X}]$.
\end{proof}
Moreover, any transposition $\tau_{k,l}$ defines an equivalence class on the monomials of $f$ by saying $\underline{X}^\alpha \equiv \underline{X}^\beta$ if $\tau_{k,l}(\underline{X}^\alpha) = \underline{X}^\beta$. 
Then we see that for $k\neq l$ from the same column of $T$ or $S$ the monomials $X^\alpha$ and $X^\beta$ have coefficients $c \in \KK$ and $-c \in \KK$. 
In particular, the polynomial $R \in \KK[\underline{X}]$ in \cref{observation1} cannot be arbitrary.

\section*{Calculation of $C$ in \cref{eq:C} in \cref{lem:cover implies ideal inclusion}}
\begin{proposition}\label{prop:Ccalculataion}
The coefficient $C$ in \cref{eq:C} in the proof of \cref{lem:cover implies ideal inclusion} equals $b!$.
\end{proposition}
\begin{proof}
We have
$$Q_2(\underline{X})=\Delta_{B_1}(\underline{X}^2)\Delta_{A \cup B_2}(\underline{X}^2) \prod_{i \in A \cup B_2} X_i=\Delta_{B_1}(\underline{X}^2)\cdot\Delta_{B_2}(\underline{X}^2)\cdot\prod_{i\in B_2}(X_1^2-X_i^2)\cdot X_1\cdot\prod_{i\in  B_2} X_i$$
and so 
\begin{align*}
    \sum_{\sigma \in S_{A \cup B_1}}\operatorname{sgn}(\sigma)\sigma(Q_2 X_1)&=\sum_{\sigma \in S_{A \cup B_1}}\operatorname{sgn}(\sigma)\sigma(\Delta_{B_1}(\underline{X}^2)\cdot\Delta_{B_2}(\underline{X}^2)\cdot X_1^2\cdot\prod_{i\in B_2}(X_1^2-X_i^2)\cdot\prod_{i\in  B_2} X_i)\\
    &=\Delta_{B_2}(\underline{X}^2)\cdot\prod_{i\in  B_2} X_i\cdot\sum_{\sigma \in S_{A \cup B_1}}\operatorname{sgn}(\sigma)\sigma(\Delta_{B_1}(\underline{X}^2)\cdot X_1^2\cdot\prod_{i\in B_2}(X_1^2-X_i^2)).
\end{align*}
In order to see that this is a multiple of $P(\underline{X})=\Delta_{A \cup B_1}(\underline{X}^2)\Delta_{B_2}(\underline{X}^2)\prod_{i \in B_2}X_i$ we compare the roots. $P$ is zero in the following cases:
\begin{itemize}
    \item $X_i=0$ for some $i\in B_2$. In this case our sum is 0 as well because of the factor $\prod_{i\in  B_2} X_i$.
    \item $X_i^2=X_j^2$ for some $i,j\in B_2$. 
    In this case our sum is 0 as well because of the factor $\Delta_{B_2}(\underline{X}^2)$.
    \item $X_i^2=X_j^2$ for some $i,j\in A\cup B_1$. 
    In this case, we show that our sum is 0 as well because the factor $\sum_{\sigma \in S_{A \cup B_1}}\operatorname{sgn}(\sigma)\sigma(\Delta_{B_1}(\underline{X}^2)\cdot X_1^2\cdot\prod_{k\in B_2}(X_1^2-X_k^2))$ is 0. 
    Observe that a summand $\operatorname{sgn}(\sigma)\sigma(\Delta_{B_1}(\underline{X}^2)\cdot X_1^2\cdot\prod_{k\in B_2}(X_1^2-X_k^2))$ can only be non-zero if $\sigma(1)\in\{i,j\}$ (i.e. $X_{\sigma(1)}\in\{X_i,X_j\}$), because otherwise $i,j\in\sigma(B_1)$ and therefore $\sigma(\Delta_{B_1}(\underline{X}^2))=\Delta_{\sigma(B_1)}$ is 0 on $X_i^2=X_j^2$. \\
    Let now $\sigma(1)=i$. In this case a summand might be non-zero. However, applying the transposition $(i~j)$ to the summand will change the sign but leave the rest unchanged, since $X_i$ and $X_j$ only appear quadratically and $X_i^2=X_j^2$. Adding all summands will therefore yield 0 as well.
\end{itemize}
Since $\sum_{\sigma \in S_{A \cup B_1}}\operatorname{sgn}(\sigma)\sigma Q_2 X_1 =0$ or $\deg(P)=\deg(\sum_{\sigma \in S_{A \cup B_1}}\operatorname{sgn}(\sigma)\sigma Q_2 X_1 )$, the sum has to be a scalar multiple of $P$, i.e. equal to $C\cdot P$ for some $C\in\RR$. 
With respect to the lexicographic monomial order, the polynomial $P$ is monic.
Thus, $C$ is the leading coefficient of the sum. 
We will ignore the common factor $\Delta_{B_2}(\underline{X}^2)\cdot\prod_{k\in  B_2} X_k$, as it does not change the leading coefficient. \\
In the leading monomial, the exponent of $X_1$ is maximal and only those summands contribute for which $\sigma(1)=1$, because in each summand the monomial $X_{\sigma(1)}$ has maximal exponent: \\
The variable $X_i$, for $i\in\sigma(B_1)$, appears only in $\sigma(\Delta_{B_1}(\underline{X}^2))$. 
It appears in $b-1$ factors, namely in one for each element in $\sigma(B_1)\setminus \{i\}$. 
It always appears quadratically, so its maximal exponent will be $2(b-1)$. 
On the other hand, $X_{\sigma(1)}$ appears in the factor $\sigma(X_1^2)$ and in the factor $\sigma(\prod_{k\in B_2}(X_1^2-X_k^2))$. 
$B_2$ is a set of size $b-1$, so the exponent of $X_{\sigma(1)}$ will be $2+2(b-1)=2b>2(b-1)$. \\
This means there are $\vert B_1\vert!= b!$ summands contributing to the leading monomial, namely those summands corresponding to permutations from $S_{B_1}$, because $\sigma(1)=1$ is the only determining condition. 
For example $\operatorname{id}
\in S_{B_1}$ contributes and yields the summand
$$
    \Delta_{B_1}(\underline{X}^2)\cdot X_1^2\cdot\prod_{k\in B_2}(X_1^2-X_k^2)=X_1^{2b}\cdot X_2^{2b-2}\cdot\ldots\cdot X_{b}^2\cdot X_{b+1}^0+R,
$$
isolating the leading monomial. This monomial also exists after applying $\sigma\in S_{B_1}$, possibly with different sign. 
So let us consider the leading monomial of $$\sigma(\Delta_{B_1}(\underline{X}^2)\cdot X_1^2\cdot\prod_{k\in B_2}(X_1^2-X_k^2))=\sigma(\Delta_{B_1}(\underline{X}^2))\cdot X_1^2\cdot\prod_{k\in B_2}(X_1^2-X_k^2).$$
For each factor $(X_i^2-X_j^2)$, $X_i^2$ contributes to the leading monomial if $i<j$, and otherwise $(-X_j^2)$ does. Therefore we get one factor $(-1)$ for each pair $i<j$ such that $\sigma(i)>\sigma(j)$. This is exactly the definition of $\sgn(\sigma)$. Thus, the monomial $X_1^{2b}\cdot X_2^{2b-2}\cdot\ldots\cdot X_{b}^2\cdot X_{b+1}^0$ appears in each summand of $\sum_{\sigma\in S_{B_1}}\sgn(\sigma)\cdot\sigma(\Delta_{B_1}(\underline{X}^2)\cdot X_1^2\cdot\prod_{k\in B_2}(X_1^2-X_k^2))$ with positive sign and therefore has the coefficient $C=b!$.
\end{proof}

\section{} \label{appendix:B Code}
We used the following SAGE \cite{sagemath} code to define Specht ideals and to explore the inclusion relation of the ideals and their radicality.
\begin{tiny}
\begin{lstlisting}[language=Python,caption={Code to calculate Specht ideals}]
n = 6
R = PolynomialRing(QQ,n,'x')
Sn = SymmetricGroup(range(n))

# The function returns the set of diparititions of n as a list
def dipartitions(n):
    ls = []    
    for k in range(int(n/2+1)):
        for p in list(Partitions(k)):
            for q in list(Partitions(n-k)):
                if p != q:
                    if (q,p) not in ls:
                        ls.append((p,q))
                if p == q:
                    ls.append((p,q,'+'))
                    ls.append((p,q,'-'))    
    return ls

# The function returns a S-Specht polynomial for a partition p  written as a list
def SnSpechtPolynomial(p,n=0):    
    if p == []:
        return R(1)    
    q = Partition(p).conjugate()
    f = R(1)
    r = n
    for l in range(len(q)):
        for i in range(q[l]-1): 
            for j in range(i+1,q[l]):
                f *= R.gens()[r+i]-R.gens()[r+j]
        r += q[l]    
    return f

# The function returns a B-Specht polynomial using the S-Specht polynomial function
def BnSpechtPolynomial(b):
    f = R(1)    
    if len(b[0]) > 1:
        f *= SnSpechtPolynomial(b[0])(*(i^2 for i in R.gens()))    
    if len(b[1]) > 1:
        f *= SnSpechtPolynomial(b[1],sum(b[0]))(*(i^2 for i in R.gens()))    
    for i in range(sum(b[0]),sum(b[0])+sum(b[1])):
        f *= R.gens()[i]    
    return f

# The function returns a B-Specht polynomial to (\mu,\lambda) for input (\lambda,\mu)
def BnSpechtPolynomialReverse(b):
    f = R(1)    
    if len(b[0]) > 1:
        f *= SnSpechtPolynomial(b[0])(*(i^2 for i in R.gens()))    
    if len(b[1]) > 1:
        f *= SnSpechtPolynomial(b[1],sum(b[0]))(*(i^2 for i in R.gens()))    
    for i in range(sum(b[0])):
        f *= R.gens()[i]    
    return f
    

# If \lambda \neq \mu the function returns B-Specht polynomials to 
# (\lambda,\mu) and (\mu,\lambda)
# If \lambda = \mu it the functions returns the pair of D-Specht polynomials to
# {\lambda,+} and {\lambda,-}
def DnSpechtPolynomials(b):
    if b[1] != b[0]:
        c = [b[1],b[0]]        
        f1 = BnSpechtPolynomial(b) 
        f2 = BnSpechtPolynomial(c)        
        return [f1,f2]    
    else:
        f1 = BnSpechtPolynomial(b)
        f2 = BnSpechtPolynomialReverse(b)    
        return[f1+f2,f1-f2]

# Defines the ideal of D-Specht polynomials to a dipartition b as 
# the symmetric group orbit of the previously defined D-Specht polynomials
# Return is the D-Specht ideal if \lambda \neq \mu
# Return is the pair of D-Specht ideals for \pm if \lambda = \mu
def DnSpechtIdeal(b):
    Sn = SymmetricGroup(range(sum(b[0])+sum(b[1])))
    I1 = ideal([DnSpechtPolynomials(b)[0](*i(R.gens())) for i in Sn])
    I2 = ideal([DnSpechtPolynomials(b)[1](*i(R.gens())) for i in Sn])    
    if b[1] != b[0]:    
        I = I1 + I2
        return I    
    else:  
        return [I1,I2]
\end{lstlisting}
\end{tiny}
\end{appendix}

\end{document}